\documentclass{amsart}
\usepackage{amsmath,amsthm,amsfonts,amssymb,amscd,mathrsfs, stmaryrd}
\usepackage[left=3.5cm,right=3.5cm]{geometry}
\usepackage{graphicx}
\usepackage{color}

\usepackage{bbm}
\usepackage{comment}
\usepackage{tikz}
\usepackage{siunitx}
\usepackage{subfigure}

\usepackage{psfrag}

\usepackage{epic,eepic}
\usepackage{ebezier}

\usepackage{mathtools}      
\usepackage{bm}             
\usepackage{indentfirst}    
\usepackage{mathdots}
\usepackage{thmtools}
\usepackage{enumitem}       
\usepackage[colorlinks=true,backref=page]{hyperref} 
\usepackage[font=small]{caption}
\usepackage{tikz-cd}        
\usetikzlibrary{decorations.pathreplacing}
\usepackage{tikz}
\usetikzlibrary{datavisualization.formats.functions}

\usetikzlibrary{arrows,shapes,positioning}
\usetikzlibrary{decorations.markings}
\tikzstyle arrowstyle=[scale=1]
\tikzstyle directed=[postaction={decorate,decoration={markings,mark=at position .5 with {\arrow[arrowstyle]{stealth}}}}]
\tikzstyle reverse directed=[postaction={decorate,decoration={markings,mark=at position .5 with {\arrowreversed[arrowstyle]{stealth};}}}]
\usetikzlibrary{calc}

\theoremstyle{plain}
\newtheorem{theorem}{Theorem }

\newtheorem{prop}{Proposition}

\newtheorem{cor}{Corollary}
\newtheorem{remark}{Remark}

\newtheorem{que}{Question}
\newtheorem{Con}{Conjecture}

\long\def\begcom#1\endcom{}

\newcommand{\diam}{\operatorname{diam}}

\newcommand{\graph}{\operatorname{graph}}
\newcommand{\length}{\operatorname{\length}}

\newcommand{\Diff}{\operatorname{Diff}}

\newcommand{\vep}{\varepsilon}

\def\Orb{\operatorname{Orb}}
\def\Diff{\operatorname{Diff}}

\def\supp{\operatorname{supp}}

\def\length{\operatorname{length}}

\def\ln{\operatorname{ln}}

\def\top{\operatorname{top}}
\def\Int{\operatorname{Int}}
\def\dim{\operatorname{dim}}

\def\vep{\varepsilon}

\usepackage{graphicx}

\begin{document}

\title{Metric entropy and  homoclinic  growth rate }

\author[Gang Liao and Jing Wei]{Gang Liao and Jing Wei}

\email{lg@suda.edu.cn}

\email{weijing1910@163.com}

\email{}

\date{}

	\thanks{2010 {\it Mathematics Subject Classification}. {37C29, 37D25, 37B10, 37C35}}

\keywords{homoclinic point;   metric entropy; Newhouse domain; super-exponential  growth}

\thanks{}

\thanks{School of Mathematical Sciences, Center for Dynamical Systems and Differential Equations, Soochow University,
	Suzhou 215006,  China;  This research  was partially supported by the National Key Research and Development Program of China (2022YFA1005802),   NSFC (12071328,  12122113).  Jing Wei  is the corresponding  author}

\maketitle



\begin{abstract} 
In this paper, we investigate the relationship between chaos and homoclinic orbits from a quantitative perspective. Let $f$  be a $C^r$  diffeomorphism ($r > 1$) on a compact Riemannian manifold preserving an ergodic hyperbolic measure $\mu$. We show that the homoclinic growth rate is bounded below by the metric entropy $h_\mu(f)$:
\smallskip
\smallskip
\smallskip

\begin{itemize}
\item[] (Log-type)\quad
$\limsup_{n \to \infty} \frac{1}{n} \log \sharp H(p, n, \vep) \ge h_\mu(f)$;
\\
\item[] (Exp-type)\quad
$\limsup_{n \to \infty} \frac{\sharp H(p, n, \vep)}{e^{n h_\mu(f)}} > 0$, if $\mu$ is a measure of maximal entropy.
\end{itemize}
\smallskip
\smallskip
\smallskip
Here, $p$ is any hyperbolic periodic point homoclinically related to $\mu$, and $H(p, n, \vep)$ denotes the set of $\vep$-order $n$  homoclinic points.

This result generalizes Mendoza’s work \cite{ML} from surfaces to higher-dimensional systems from a measure-theoretic viewpoint. We also examine the sharpness of this estimate by demonstrating that, in the Newhouse domain, $C^r$-generic diffeomorphisms exhibit a superexponential growth in the number of homoclinic points.

\end{abstract}

\allowdisplaybreaks

\section{Introduction}\label{sec:intro}
The homoclinic point is an important concept in dynamical systems.
In the late 19th century, while studying the restricted three-body problem, H. Poincar\'e first observed the existence of transverse homoclinic orbits.
This structure leads to the emergence of infinitely many periodic points, thereby inducing extreme complexity in the system's behavior.
In the 1960s, Smale \cite{Sm} introduced the horseshoe map.
It was proven that the existence of transverse homoclinic point is equivalent to the existence of the horseshoe in the neighborhood of this transverse homoclinic point by the Birkhoff-Smale Theorem. Consequently, the presence of a transverse homoclinic point implies that the system has positive topological entropy.

A famous conjecture of Palis \cite{Pa2} states: 
\begin{Con}\label{Palis}Any diffeomorphism can be $C^r$  approximated by a Morse–Smale one or by one
	exhibiting a transversal homoclinic orbit.
\end{Con}

 Morse–Smale systems admit very simple dynamics: the non-wandering 
set is hyperbolic and finite,  thus the system has zero entropy.  
It exhibits that the complicated dynamics are concentrated in the domain related to transverse homoclinic orbits: 
$$\overline{\{f:\,\,f\,\,\text{has positive topological entropy}\}}= \overline{\{f:\,f\,\,\text{has  transverse homoclinic orbits}\}}.$$
The Conjecture \ref{Palis} is shown to hold for  $C^1$ diffeomorphisms (Bonatti-Gan-Wen \cite{BGW} in dimension 3 and Crovisier \cite{SC} in any dimension).  
This Conjecture provides the qualitative description: chaos arises from transverse homoclinic orbits in the perturbation sense.  A natural question follows: 
\begin{que}\label{entropy and hom} How to establish a quantitative relationship between chaos and transverse homoclinic orbits? 
\end{que}
If the above question is properly solved, then the complexity of the system can be quantitatively analyzed by counting the number of transverse homoclinic orbits.

In dynamical systems, entropy measures the complexity of the orbits.
Systems with positive entropy exhibit chaotic dynamics.
As noted, chaos arises from transverse homoclinic points under permissible perturbation.
As a homoclinic orbit contains infinitely many points, a precise counting method is essential.

In order to count transverse homoclinic points, Mendoza \cite{ML} introduced the concept of  {\em $\vep$-order-$n$}  homoclinic points,  which remain outside $\vep$-neighborhood of the periodic point under iterations within time $n$.    If the homoclinic point stays $\vep$-close to the periodic point, it does not contribute complexity at that time. 
It is shown that if the homoclinic class (the closure of transverse homoclinic points) is   {\bf uniformly  hyperbolic}, then the topological entropy of the homoclinic class equals the exponential growth rate of homoclinic points of $\varepsilon$-order $n$. 
When the homoclinic class is non-uniformly hyperbolic, the equality is established {\bf on surface}.  
In the high dimensional case, it is still unknown. 
In this paper, we go beyond the uniform hyperbolicity and aim to verify  Question  \ref{entropy and hom} {\bf in  high dimensions}.

By the variational principle,  the topological entropy is approximated by the metric entropy of ergodic measures. For surface diffeomorphism with positive topological entropy, approximating ergodic measures also have positive metric entropy and thus are necessarily hyperbolic (i.e., all Lyapunov exponents are nonzero).   

In high dimensions, we give the precise relationship between the metric entropy of hyperbolic measure and the growth of homoclinic points.  For the hyperbolic ergodic measure, a challenge is to define its homoclinic points since it is not generally supported on a periodic orbit.  As the recent work of Buzzi-Crovisier-Sarig \cite{BCS} investigating the homoclinic relation of measures, we study the homoclinic growth of hyperbolic measures by considering the periodic orbits homoclinically related to the measure.

Our first result is:

\begin{theorem}[log-type]\label{theorem a}
	Let $f$ be a $C^{r}$ diffeomorphism with $r>1$ on a compact Riemannian manifold $M$ and $\mu$ be a hyperbolic ergodic measure with metric entropy $h_{\mu}(f)$. 
	Then for small $\varepsilon >0$,
	$$
	\limsup _{n\to \infty } \frac{1}{n} \log \sharp H(p,n,\varepsilon ) \ge  h_{\mu}(f),
	$$where $p$ is any hyperbolic periodic point homoclinically  related to $\mu$ and $H(p,n,\varepsilon) $ is the set of homoclinic points of $\vep$-order-$n$. 
\end{theorem}

Theorem \ref{theorem a} provides a log-type homoclinic growth for any hyperbolic ergodic measure.  
When $f$ admits a hyperbolic measure of maximal entropy (i.e., the topological entropy $h_{\top}(f)$), we obtain an exp-type estimate:

\begin{theorem}[exp-type]\label{theorem b}
	Let $f$ be a $C^{r}$ diffeomorphism with $r>1$ on a compact Riemannian manifold $M$ and $\mu$ be a hyperbolic ergodic measure of maximal entropy. 
	Then for any hyperbolic periodic point $p$ homoclinically related to $\mu$, and small $\varepsilon>0$,
	$$
	\underset{n\to\infty}{\limsup}\frac{\sharp H(p,n,\varepsilon)}{e^{nh_{\mu}(f)}} >0.
	$$
\end{theorem}

By a result of Newhouse \cite{N1}, $C^{\infty}$ diffeomorphisms on a compact manifold admit measures of maximal entropy.
In particular, for surface diffeomorphisms with positive topological entropy,  the measures of maximal entropy  are necessarily hyperbolic.  Therefore, we have: 
\begin{cor}
	Let $f$ be a $C^{\infty}$ diffeomorphism on a compact surface $M$ with $h_{\top}(f)>0$. 
	Then there exists a hyperbolic point $p$ and small $\varepsilon>0$,
	$$
	\underset{n\to\infty }{\limsup}\frac{\sharp H(p,n,\varepsilon)}{e^{nh_{\top}(f)}} >0.
	$$
\end{cor}

The above results provide lower-bound estimates for the growth rate of transverse homoclinic points. 
Naturally, we would like to know what happens with its upper-bound estimates.  
We will show that in general the upper bound may be super-exponential.  
This leads us to the analysis
of dynamics in the Newhouse domain. 

\noindent \textbf{Newhouse domain}
As previously stated, transverse homoclinic points contribute to chaos. 
Homoclinic tangencies, on the other hand, may trigger bifurcations or parameter-sensitive phenomena. 
Homoclinic tangencies undermine structural stability, causing drastic changes in dynamical behavior under small parameter variations. 
For example, a horseshoe may emerge near a tangency point.   

Given $2\le r \le \infty$, let $C^{r}(M)$ denote the space of $C^{r}$ maps on a compact manifold $M$ with the uniform $C^{r}$-topology, and let $\Diff^{r}(M)$ be the space of the $C^{r}$ diffeomorphisms with the same topology.
For a compact surface $M$, the \textit{Newhouse domain}  \cite{N} is defined as:
$$
\mathcal{N}^r=\Int \overline{(\left \{ f\in \Diff^{r}(M) \text{ with homoclinic tangencies} \right \} )}, \, \quad r\ge 2.$$
Palis \cite{Pa1} conjectured that the space of two-dimensional $C^{r}$ diffeomorphisms
with $r \ge 2$ is the closure of the union of just two open sets: axiom A systems, and the Newhouse domain. 
The Newhouse domain has been applied to estimate the upper bounds of the super-exponential growth of periodic points on two-dimensional manifolds \cite{Kal}.

Artin and Mazur \cite{AM} showed that there exists a dense set in $C^{r}(M)$ where the number of period-$n$ points grows at most exponentially.
On the contrary, in a generic sense (i.e., in a countable intersection of open and dense sets),  
Kaloshin \cite{Kal} showed that the number of periodic points grows faster than any preassigned rate 
 in the Newhouse domain.
Berger \cite{Be} showed for any $2\le r \le \infty$ and the dimension of $M$ at least $2$, there exists an open set $U$ of $C^{r}$ self-mapping such that for a generic map $f$ from $U$ the number of periodic points also grows as fast as asked. 

In this paper,  we establish the fast growth result for homoclinic points.

\begin{theorem}\label{theorem c}

Let $\mathcal{N} \subset \Diff^{r}(M)$ be the  Newhouse domain with a compact surface $M$ and $2\le r \le \infty$. 
Then for an arbitrary sequence of positive integers $\left \{ a_{n} \right \} _{n=1}^{\infty } $  there exists a residual set $\mathcal{R} _{a}$ depending on the sequence $\left \{ a_{n} \right \} _{n=1}^{\infty } $ with the property that $f\in \mathcal{R} _{a}$ implies that 
$$
\limsup_{n\to \infty } \frac{1}{a_{n}} \sharp H(p,n,\varepsilon) = \infty,
$$
for some hyperbolic periodic point $p$ of $f$ and small $\vep>0$.
\end{theorem}

\begin{remark} In \cite{Kal, Be}, the super-exponential growth of periodic points is established. These points are located on the center manifold and exhibit Lyapunov exponents close to zero, and are therefore not homoclinically related to one another. A key challenge in Theorem \ref{theorem c}, however, is the selection of homoclinic points that not only exhibit a fast growth rate, but are also homoclinically related to the same periodic orbit. This contributes a new complicated phenomenena in Newhouse domain.

\end{remark} 

\bigskip

\noindent\textbf{More on the counting problem.} In addition to homoclinic orbits, 
periodic orbits are another class of special orbits that are widely studied. 
Research on counting periodic orbits provides inspiration for counting transverse homoclinic orbits. 
 
Let $f:M\to M$ be a diffeomorphism on a compact Riemannian manifold $M$, and define 
$$
P_{n}(f):=\left\{x\in M : x=f^{n}(x) \right\}.
$$
Let $ \mathcal{M} ^{*}(f)$ denote the set of all hyperbolic ergodic invariant measures.

Bowen's definition of topological entropy suggests a relationship with the growth rate of periodic orbits.
When $f$ is axiom A diffeomorphism, 
$$
h_{top}(f) =\underset{n\to\infty}{\lim \sup}\frac{\log \sharp P_{n}(f)}{n},
$$
see \cite{Bo}.
Katok \cite{Kat} proved that if $f$ is $C^{r}$ ($r>1$) diffeomorphism on a compact surface with $h(f)>0$,
$$ 
\underset{n\to\infty}{\lim \sup}\frac{\log \sharp P_{n}(f)}{n}\ge h_{top}(f),
$$
and more generally, if $f$ is $C^{r}$ ($r>1$) diffeomorphism and $\mu \in \mathcal{M} ^{*}(f)$,
$$ 
\underset{n\to\infty}{\lim \sup}\frac{\log \sharp P_{n}(f)}{n}\ge h_{\mu}(f).
$$

Katok \cite{K1} conjectured the  following exp-type  growth of periodic orbits for surface diffeomorphisms. 
\begin{Con}\label{Katok}Is it true that
	\begin{align}\label{1}
	\underset{n\to +\infty }{\limsup}\frac{\sharp P_{n}(f)}{e^{nh_{top}(f)}}>0.
	\end{align}
	for any (i) $C^{r}$ ($r>1$)
	or (ii) $C^{\infty}$ 
	diffeomorphism $f$ of a compact surface?
\end{Con}
In \cite{S}, Sarig proved that for any $C^{r}$ ($r>1$) surface diffeomorphism with positive entropy, there exists a countable Markov coding and further used it to show Conjecture \ref{Katok} for $C^{\infty}$ 
	surface diffeomorphisms (or more generally, $C^{r}$ surface diffeomorphisms with a measure of maximal entropy).
Ben Ovadia \cite{SBO} later  generalized (\ref{1}) to high diemsnion by considering diffeomorphisms with a hyperbolic measure of maximal entropy.
Inspired by this, we study the relationship between topological entropy and transverse homoclinic points in Theorem \ref{theorem b}, which establishes a homoclinic version of Conjecture \ref{Katok}.

What's more, we would like to mention the upper bound estimate or equality for the growth rate of periodic orbits.  For prime orbits, define: 
$$
P_{n}^{*}(f):=\left\{\Orb(x) : x=f^{m}(x),\,\, m\le n \right\}.
$$
By results of Margulis\cite{MG}, and Parry-Pollicot\cite{PM},  for mixing axiom A systems:
\begin{eqnarray}\label{pri}
    \underset{n\to\infty}{\lim }\frac{\sharp P^{*}_{n}(f)}{e^{nh_{top}(f)}/nh_{top}(f)}=1.
\end{eqnarray}
This equality reflects the correspondence between entropy and the exp-type  growth of periodic orbits.
It originates from the famous Prime Number Theorem in number theory: 
$$
\underset{n\to\infty}{\lim }\frac{\pi(n)}{n/ \log n}=1,
$$
where $\pi(n)$ is the number of primes $\le n$.
Prime numbers are the atoms of natural numbers.
By the Closing Lemma and specification, general orbits can be understood through the approximation and combination of periodic orbits, which can thus be regarded as the ``dynamical atoms" of dynamical systems.
Both are indecomposable basic units in their respective fields.
Thus, equality (\ref{pri}) is also called the Prime orbit Theorem.

In the setting of non-uniformly hyperbolic systems, Hirayama \cite{CH} considered, for a $C^{r}$ ($r>1$) surface diffeomorphism $f$ and numbers $\eta$, $\lambda>0$, the set $P_{n}(\eta,\lambda)$ of all period $n$ points satisfying $\left \| Df^{\pm k} (f^{i}p) \right \| \ge\eta e^{k\lambda } $ for all $j\ge 0$ and $0\le i\le n-1$. 
He showed that the topological entropy equals the exponential growth rate of the cardinality of $P_{n}(\eta,\lambda)$ for some $\eta, \lambda$.
Newhouse \cite{N1} proposed hyperbolic rate to give the log-type equality for  hyperbolic  measure of maximal entropy. Liao et al. \cite{LST, LSY} further generalized it to any ergodic and non-ergodic hyperbolic measures.   Burguet \cite{Bur} considered $C^{\infty}$ surface diffeomorphisms  and for any $\delta \in (0,h_{top}(f))$, established   the log-type equality for periodic points with Lyapunov exponents $\delta$-away from zero by showing asymptotical  periodic expansiveness.    

This paper is organized as follows. 
Section \ref{sec:pre} introduces the necessary definitions. 
In Section \ref{log} we show the log-type estimate for any ergodic hyperbolic measure:  we first provide preliminaries for Theorem \ref{theorem a}.
In Section \ref{hc}, 
for any periodic point $p$ homoclinically related to $\mu$,
we choose a periodic point $q$ in a horseshoe homoclinically related to $p$, and use order-$n$ homoclinic points of $q$ to count transverse homoclinic points of $p$. In Section \ref{exp}, we show the exp-type estimate for ergodic hyperbolic measure of maximal entropy: 
 Section \ref{exp1}  presents preliminaries for Theorem \ref{theorem b}. 
The proof of Theorem \ref{theorem b}  is no longer to directly study the transverse homoclinic points on the manifold $M$, but to analyze the transverse homoclinic points in symbolic space.
In Section \ref{log} we show the super-exponential estimate  in Newhouse domain:  we use another periodic point $q$ rather than the initial one  $p$ but homoclinically related to $p$ to obtain an interval of tangencies,
which is then perturbed to generate infinitely many transverse homoclinic points of arbitrarily large order.

\section{Preliminaries}\label{sec:pre}

\subsection{Homoclinic relation and Homoclinic point of order-$n$}

 Let $M$ be a compact Riemannian manifold and let $f:M\to M$ be a diffeomorphism. Given a hyperbolic periodic point $p$,  
the \textit{ local stable (unstable) manifold } of size $\varepsilon$ is defined as 
$$
W_{\varepsilon}^{s}(p)=\left\{ x\in M:d(f^{n}x,f^{n}p) \le \varepsilon,\forall n\ge0\right\},
$$
$$
W_{\varepsilon}^{u}(p)=\left\{ x\in M:d(f^{-n}x,f^{-n}p)\le \varepsilon,\forall n\ge0\right\}.
$$
The \textit{stable (unstable) manifold} is defined as
$$
W^{s}(p)=\left\{ x\in M:d(f^{n}x,f^{n}p)\to 0, n\to +\infty \right\},
$$
$$
W^{u}(p)=\left\{ x\in M:d(f^{-n}x,f^{-n}p)\to 0, n\to +\infty \right\}.
$$
A point $x\in W^{s}(p)\cap W^{u}(p)\setminus \left\{p\right\}$ is called a homoclinic point for $p$. 
The point $x$ is called a \textit{transverse homoclinic point}, if the intersection of $W^{s}(p)$ and $W^{u}(p)$ is transverse at $x$, i.e.
$$
T_{x} M=T_{x}W^{s} (p)\oplus T_{x}W^{u} (p).
$$
\begin{figure}[h]
	\centering
	\includegraphics[width=0.3\linewidth]{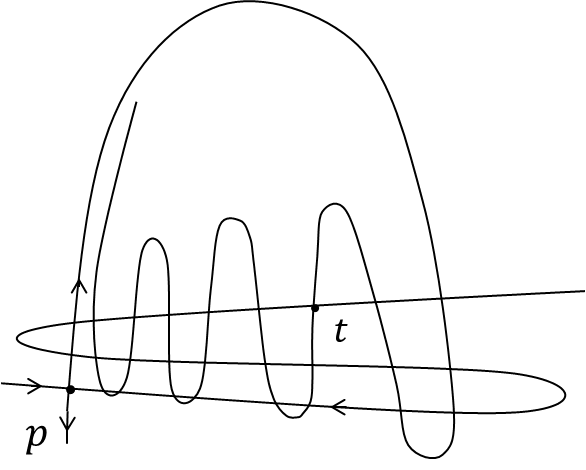}
	\caption{The transversely homoclinic point $t$ of periodic point $p$. }
	\label{TH}
\end{figure}
Denote 
$$
\Gamma (p):=\left\{ x\in M:\text{ $x$ is a transverse  homoclinic point for $p$} \right\}.
$$

For small $\varepsilon>0$, if $\Gamma(p)\ne \emptyset$, for $x\in \Gamma(p)$,  let 
\begin{align}
\Theta ^{s}(x,p,\varepsilon)={}&\min \left\{n:f^{n}(x)\in W_{\varepsilon}^{s}(p) \right\},\notag\\[2mm]
\Theta ^{u}(x,p,\varepsilon)={}&\min \left\{n:f^{-n}(x)\in W_{\varepsilon}^{u}(p)\right\},\notag\\[2mm]
\Theta (x,p,\varepsilon)={}& \Theta ^{s}(x,p,\varepsilon)+\Theta ^{u}(x,p,\varepsilon),\notag 
\end{align}
A point $x\in \Gamma(p)$ is called  \textit{$\varepsilon$-order  $n$} if  $$ \Theta ^{s}(x,p,\varepsilon)+\Theta ^{u}(x,p,\varepsilon) =n.$$
\begin{figure}[h]
	\centering
	\includegraphics[width=0.3\linewidth]{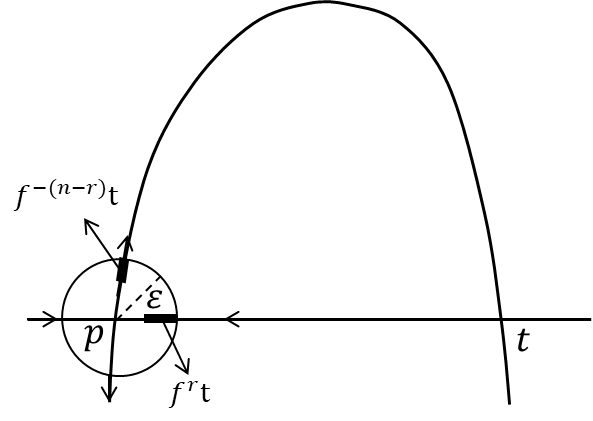}
	\caption{Homoclinic point $f^{-(n-r)}t$ of $\varepsilon$ order-$n$ of the periodic point $p$}
	\label{order}
\end{figure}
It is clear that if $\Theta ^{s}(x,p,\varepsilon)=n_{1}$ and $\Theta ^{u}(x,p,\varepsilon)=n_{2}$, these imply that the point $x$ is outside a neighborhood of size $\varepsilon$ of the periodic point $p$ in $W^{u}(p)$ in the past $n_{2}$ time and,  outside a neighborhood of size $\varepsilon$ of the periodic point $p$ in $W^{s}(p)$ in the future $n_{1}$ time. 
Let 
$$
H(p,n,\varepsilon):=\left\{x\in \Gamma(p):\Theta ^{s}(x,p,\varepsilon)= n,\Theta ^{u}(x,p,\varepsilon)=0\right\}.
$$ 
This is a set of the points of the $\varepsilon$-order $n$.
See Fig.\ref{order}.

We say that two hyperbolic periodic points $p_{1},p_{2}$ are  \textit{homoclinically related}, if $W^{s}(p_{1}) $ intersects $W^{u}(p_{2}) $ transversally and $W^{u}(p_{1}) $ intersects $W^{s}(p_{2}) $ transversally,
write $p_{1}\overset{h}{\sim } p_{2}$.
It is clear that this is an equivalence relation.

From \cite{BCS}, the notion of homoclinic relation extends to measures. Denote by  $\mathcal{M} ^{*}(f)$ the set of hyperbolic ergodic measures.  Let $f$ be a $C^{1+\alpha}$ diffeomorphism, by Pesin's stable manifold theorem, for $\mu \in \mathcal{M} ^{*}(f)$,  almost every $x\in M$, $x$ has stable and unstable manifold. 
For $\mu_{1},\mu_{2}\in \mathcal{M} ^{*}(f)$, $\mu_{1},\mu_{2}$ are \textit{homoclinically related}, if there are measurable sets $A_{1},A_{2}\subset M$ with $\mu_{i}(A_{i})>0$ such that for all $(x_{1},x_{2})\in A_{1}\times A_{2}$, the manifold $W^{u}(x_{1})$ and $W^{s}(x_{2})$ have a point of transverse intersection, and there are measurable sets $B_{1},B_{2}\subset M$ with $\mu_{i}(B_{i})>0$ such that for all $(y_{1},y_{2})\in B_{1}\times B_{2}$, the manifold $W^{s}(y_{1})$ and $W^{u}(y_{2})$ have a point of transverse intersection.
We write $\mu_{1}\overset{h}{\sim } \mu_{2}$. 
The homoclinic relation of measures is an equivalence relation.
We write the sets of hyperbolic periodic orbits as 
$$
P_{h}(f)=\left\{ \Orb( x) \subset M: x \text{ is a hyperbolic periodic point} \right\}.
$$
Each $ o\in P_{h}(f)$ carries a unique invariant ergodic measure $\mu_{o}$.  We  denote  $p\overset{h}{\sim } \mu$ when $\mu_{o}\overset{h}{\sim }\mu$ and  $o$ is the periodic orbit of $p$.

We define the set $W^{u}(o) \pitchfork W^{s}(x)$ ($W^{s}(o) \pitchfork W^{u}(x) $) as the set of transverse intersection points.
For $\mu \in \mathcal{M} ^{*}(f)$, one writes $o\overset{h}{\sim }\mu$ when $\mu_{o}\overset{h}{\sim }\mu$.
This is equivalent to be require both $W^{u}(o) \pitchfork W^{s}(x) \ne \emptyset $ and $W^{s}(o) \pitchfork W^{u}(x) \ne \emptyset $ for $\mu$-almost every $x$. 
We also write $o\overset{h}{\sim } \mu$  as $p\overset{h}{\sim } \mu$, if $o$ is the periodic orbit of $p$.

 
\section{log-type growth}\label{log}

We now proceed to the proofs of our main results, beginning with the log-type estimate.
\subsection{Pesin set}
Given $\lambda  \gg \epsilon  >0$, and for $k\ge 1$ with $k\in\mathbb{Z}$, define
$
\Lambda _{k} (\lambda, \epsilon )
$
to be the set of all points $x\in M$ for which there is a splitting 
$
T_{x} M=E^{s} (x)\oplus E^{u} (x)
$
with invariant property 
$
D_{x} f^{t} (E^{s} (x))=E^{s} (f^{m}(x)),D_{x} f^{t} (E^{u} (x))=E^{u} (f^{t}(x))
$
and satisfying:
\begin{itemize}\item[(1)]$\left \| Df^{n} \mid _{E^{s}(f^{t}(x)) }  \right \| \le e^{\epsilon  k} e^{-(\lambda -\epsilon  )n}e^{\epsilon  \left | t \right | },n\ge 1,t\in \mathbb{Z} $;\\
\item[(2)]$\left \| Df^{-n} \mid _{E^{u}(f^{t}(x)) }  \right \| \le e^{\epsilon  k} e^{-(\lambda  -\epsilon  )n}e^{\epsilon  \left | t \right | },n\ge 1,t\in \mathbb{Z}$;\\
\item[(3)]$\left | \sin\angle (E^{s}(f^{t}(x)) ,E^{u}(f^{t}(x)))  \right |  >  e^{-\epsilon  k}e^{-\epsilon  \left | t \right | },t\in \mathbb{Z}$.
\end{itemize}
 Define the \textit{Pesin set} as $\Lambda (\lambda,\epsilon  )=\cup^{+\infty}_{k=1}\Lambda _{k} (\lambda,\epsilon  )$.
\subsection{Admissible manifolds}\cite{Kat, KH, P} 
Let $f:M\to M$ be a $C^{r}$ ($r>1$) diffeomorphism that has the Pesin set $\Lambda (\lambda,\epsilon )=\cup^{+\infty}_{k=1}\Lambda _{k} (\lambda,\epsilon )$.
Let $d$ be the Riemannian metric on $M$.
Assume that $\dim E^s =s$ and $\dim E^u=u$. 
For $(z_1,z_2) \in \mathbb{R}^{s}\times \mathbb{R}^{u}$, we define $|(z_1,z_2)|=\max\{|z_1|_s,|z_2|_u\}$, where $|\cdot|_s$ and $|\cdot|_u$ are the Euclidean norms on the $\mathbb{R}^{s}$ and $\mathbb{R}^{u}$ respectively.

For any $x\in \Lambda_{k}(\lambda,\epsilon )$, there exist $\varepsilon_{k}>0$ and $C^{1}$ embedding
$$
\Psi_{x}:B_{\varepsilon_{k}}^{s}\times B_{\varepsilon_{k}}^{u}\to M,
$$
with the following properties:
\begin{itemize}
    \item[(i)] $ \Psi_{x}0=x$; $D\Psi_{x}(0)$ takes $\mathbb{R}^{s}\times \{0\}$ and $\{0\}\times \mathbb{R}^{u}$ to $E^s(x)$ and $E^u(x)$ respectively;
    \item[(ii)] Let $f_x=\Psi^{-1}_{fx}\circ f \circ \Psi_x$ be the connecting map between the chart at $x$ and the chart at $fx$, defined whenever it makes sense, and let $f^{-1}_x=\Psi^{-1}_{f^{-1}x}\circ f^{-1} \circ \Psi_x$ be defined similarly. Then for $y\in B_{\vep_k}$, 
    \begin{eqnarray*}e^{\lambda -2\epsilon} |v|&\le &|Df_x(y)v|,  \,\, \text{for } v\in \{0\}\times \mathbb{R}^{u},\\
      |Df_x(y)v| &\le & e^{-\lambda +2\epsilon} |v|, \,\, \text{for } v\in  \mathbb{R}^{s} \times \{0\}.\end{eqnarray*}
     \item [(iii)] For all $y_1,y_2\in B_{\varepsilon_k}$, we have 
     $$
     Kd(\Psi_xy_1, \Psi_xy_2) \le |y_1-y_2|\le \varepsilon_k^{-1} d(\Psi_xy_1, \Psi_xy_2),
     $$
     for some universal constant $K>0$.
\end{itemize}
Here,  $B_{\vep_k}$ and $B_{\vep_{k}}^{s/u}$ are Euclidean $\vep_k$-balls around the origin in $\mathbb{R}^{s+u}$ and $\mathbb{R}^{s/u}$.
$\varepsilon_{k}$ decreases as $k$ increases.
We call $R(x,t)=\Psi _{x}(B_{\frac{t\varepsilon_{k}}{2}}^{s}\times B_{\frac{t\varepsilon_{k}}{2}}^{u}) $ \textit{the $t$-regular neighborhood of $x$} with $0< t\le 1$.

For  small $0<\gamma<\frac{1}{2}$, denote 
$$
U_{x}=\{ \Psi _{x}(\graph \phi):\,\,\phi\in C^{1}(B_{t\varepsilon_k}^{u}, B_{t\varepsilon_k}^{s}),\,\,| \phi(0)| \le \frac{t\varepsilon_{k}}{4},\,\,\left\| D \phi\right\| \le \gamma\}
$$
and 
$$
S_{x}=\{ \Psi _{x}(\graph \phi):\,\,\phi\in C^{1}(B_{t\varepsilon_k}^s, B_{t\varepsilon_k}^{u}),\,\,| \phi(0)| \le \frac{t\varepsilon_{k}}{4},\,\,\left\| D \phi\right\| \le \gamma\}.
$$

A manifold of the form $N \cap R(x,t)$ with  $N\in U_{x}$ is called an \textit{admissible $(u,t)$-manifold} near $x$.
Similarly, $N \cap R(x,t)$ with $N\in S_{x}$
is an \textit{admissible $(s,t)$-manifold} near $x$.

For any point $x$ in $\Lambda_{k}(\lambda,\epsilon )$,  we introduce the local unstable manifold at $x$ associated with $\Phi_x$ and $0<t\le 1$. 
It is defined to be the component of $W^u(x) \cap \Phi_x (B_{t\varepsilon_{k}}^{s}\times B_{t\varepsilon_{k}}^{u}) )$ containing $x$.
The $ \Phi_x ^{-1}$-image of the local unstable manifold at $x$ is the graph of a function:
$$\phi^u:B^{u}_{t\vep_k} \to  B^{s}_{t\vep_k}$$
with $\phi^u(0)=0$ and $\|D\phi^u \|\le \gamma$ for small $t$.
Similarly, the local stable manifold at $x$ is defined to be the component of $W^s(x) \cap \Phi_x (B_{t\varepsilon_{k}}^{s}\times B_{t\varepsilon_{k}}^{u} )$ containing $x$.
The $ \Phi_x ^{-1}$-image of the local stable manifold is the graph of a function:
$$\phi^s:B^{s}_{t\vep_k} \to  B^{u}_{t\vep_k}$$
with $\phi^s(0)=0$ and $\|D\phi^s \|\le \gamma$. 

Given $C^1$ map $\phi_1, \phi_2:B_{t\varepsilon_{k}}^{u}\to B_{t\varepsilon_{k}}^{s}$ such that $|\phi_1(v)|>|\phi_2(v)|$ for $|v|<t\vep_k$ and $\|D\phi_i\|\le \gamma$ ($i=1,2$), we define the set 
$$\Phi_x \{(v,u)\in (B_{t\varepsilon_{k}}^{s}\times B_{t\varepsilon_{k}}^{u}):v=\theta\phi_1(u)+(1-\theta)\phi_2(u),\,\,0<\theta\le 1\}$$
is an \textit{admissible $u$-rectangle} in $\Phi_x(B_{t\varepsilon_{k}}^{u}\times B_{t\varepsilon_{k}}^{s})$. 
Similarly define an \textit{admissible $s$-rectangle} in $\Phi_x(B_{t\varepsilon_{k}}^{s}\times B_{t\varepsilon_{k}}^{u})$.

In particular, if $\Lambda \subset M$ is uniformly hyperbolic, then the size of the regular neighborhood can be chosen as a uniform constant for all $x\in \Lambda$.

\subsection{ Closing lemma}(\cite{Kat}, Lemma S.4.13 of \cite{KH})\label{clo}
Let $f:M\to M$ be a $C^{r}$ $(r>1)$ diffeomorphism of a compact $m$-dimensional Riemannian manifold $M$, and $\Lambda (\lambda,\epsilon  )=\cup^{+\infty}_{k=1}\Lambda _{k} (\lambda,\epsilon  )$ be a Pesin set, then for any $k$ and $\varepsilon>0$ there exists a number $\beta=\beta(k,\varepsilon)>0$ such that if for some point $x\in \Lambda_{k}$ and for some integer $n$ one has 
$
f^{n}(x)\in\Lambda_{k}
$ 
and 
$
d(x,f^{n}(x))<\beta,
$
then there exists a periodic point $z=z(x)\in M$ such that 
\\(i) $f^{n}z=z;$
\\(ii) $d(f^{i}(x),f^{i}(z(x)))<\varepsilon$, $i=0,1,...,n-1;$
\\(iii) the point $z$ is a hyperbolic point for $f$, and its local stable manifold (unstable manifold) is an admissible $(s,1)$-manifold ($(u,1)$-manifold) near $x$.

\subsection{ $\lambda$-lemma}(\cite{N3})\label{clo}
Let  $p \in M$  be a hyperbolic fixed point of $f$. 
Then, for any unstable disk $B \subset W^u(p)$, any point $x \in W^s(p)$, any  $\dim W^u(p)$-disk $D$ that intersects $W^s(p)$ transversally at point $x$, and any $\varepsilon > 0$, there exists  $N$  such that if  $n > N $, then  $f^n(D)$  contains a disk that is $C^1 $ $\varepsilon $-close to $ B$.

\subsection{Katok's entropy formula}(\cite{Kat}) 
Recall Katok's definition of metric entropy. 
Let $f:M\to M$ be a diffeomorphism on a compact manifold $M$ and $\mu \in \mathcal{M} ^{*}(f)$.
For $\varepsilon>0$, $n\in\mathbb{N}$, denote
$$
D_{f}(x,n,\varepsilon)=\left\{y\in M:d(f^{i}(x),f^{i}(y))<\varepsilon,i=0,1,...,n-1\right\},
$$
which is called a $(n,\varepsilon)$-Bowen ball.
Given $0<\delta<1$, $W\subset M$ is called a $(n,\varepsilon,\delta)$ cover of $M$, if 
$$
\mu(\underset{x\in W}{\cup}D_{f}(x,n,\varepsilon))>1-\delta.
$$
We write $N_{f}(n,\varepsilon,\delta)=\min \left\{ \sharp W: W \text{ is an $(n,\varepsilon,\delta)$ cover of $M$ }\right\}$, then the metric entropy of $f$ about $\mu $ is 
\begin{align}
h_{\mu}(f)
&={}\underset{\varepsilon\to0}{\lim} \underset{n\to\infty}{\lim \sup}\frac{\log N_{f}(n,\varepsilon,\delta)}{n}\notag\\
&={}\underset{\varepsilon\to0}{\lim} \underset{n\to\infty}{ \lim \inf}\frac{\log N_{f}(n,\varepsilon,\delta)}{n}.\notag
\end{align}

\subsection{Entropy and horseshoes}
A hyperbolic set $\Upsilon$ is called \textit{locally maximal} if it has an open neighborhood $U$ such that $\Upsilon=\underset{n\in Z}{\cap } f^{n}U$.
A \textit{basic set} $\Upsilon$ for $f$ is a compact, transitive, hyperbolic, locally maximal invariant set.
A totally disconnected and infinite basic set is called a \textit{horseshoe}.  A base of a  horseshoe $\Upsilon$ is a compact set $\Upsilon_0$ such that for some $l\ge 1$, 
\begin{itemize} \item $\Upsilon =\Upsilon_0\cup \cdots\cup f^{l-1}(\Upsilon_0),\quad f^l(\Upsilon_0)=\Upsilon_0$;\\
\item $f^l: \Upsilon_0 \to \Upsilon_0$ is mixing. 
\end{itemize}
By  Theorem 3.1 of Mendoza \cite{ML} on mixing basic sets,  it holds that 
 for small $\vep>0$, 
\begin{align}\label{hyp}
h_{top}(f^l|\Upsilon_0)= \underset{n\to \infty}{\lim} \frac{1}{n} \log \sharp H_{f^l}( p,n,\varepsilon )\cap \Upsilon_0,
\end{align}
where $H_{f^l}(p,n,\varepsilon)$ is the set of $\vep$-order $n$ points of $f^l$ with respect to $p$.

\subsection{Proof of Theorem \ref{theorem a}}\label{hc}
\begin{proof}

  \textit{\textbf{Step 1.}
For any $\vartheta >0$, there exists a hyperbolic horseshoe $\Upsilon=\Upsilon_0\cup \cdots\cup f^{l-1}(\Upsilon_0)$ with base $\Upsilon_0$ and $l\ge 1$ such that:}

\begin{itemize}\item
$h_{\mu}(f^l)<h_{top}(f^l|\Upsilon_0 )+l\vartheta$,\\
    \item 
for any $p\overset{h}{\sim }\mu$ and $q\in \Upsilon_0$, we have 
$p\overset{h}{\sim }q$.
\end{itemize}

Let $\kappa(\mu)$ be the minimal number of Lyapunov exponnets of $\mu$ in absolute value. Then, taking   $0<\lambda\le \kappa(\mu)$, we have $\mu(\Lambda(\lambda, \epsilon))=1$ for some Pesin set $\Lambda(\lambda, \epsilon)=\cup_k \Lambda_k(\lambda, \epsilon)$. Take $\Lambda_k(\lambda, \epsilon)$ with positive $\mu$-measure. For a given point $x_{0} \in \supp(\mu |_{\Lambda_{k}(\lambda,\epsilon)})$, take $0<\tau < \frac{t}{4}\vep_k^2$ with $t$ assigned later. 
Let $\delta=1-\frac{1}{4}\mu(B(x_{0},\tau)\cap\Lambda_{k}(\lambda,\epsilon))$, then $\mu(B(x_{0},\tau)\cap\Lambda_{k}(\lambda,\epsilon))>1-\delta$,  where $B(x_{0},\tau)=\left \{y\in M:d(x_{0},y)<\tau \right \}$.  
Define $\triangle _{k}=B(x_{0},\tau)\cap\Lambda_{k}(\lambda,\epsilon)$.
For $m\in \mathbb{N} $, choose $ 0<\varepsilon < \min \{ \frac{1}{m}, \frac{1}{4} \vep_k^2 \}$ such that:
\begin{align}\notag
h_{\mu }(f)- \frac{1}{m}<\liminf _{n\to \infty } \frac{\log N_{f}(n,\varepsilon ,\delta )}{n}   \le \limsup _{n\to \infty } \frac{\log N_{f}(n,\varepsilon ,\delta )}{n}. 
\end{align}

Let $\xi $ be a finite measurable partition of $M$ with $\diam \xi$ $<\beta(k,\varepsilon)$ and $\xi>  \left \{ \triangle _{k},M\setminus \triangle _{k} \right \} $, where $\beta(k,\varepsilon)$ is taken as in the Closing lemma \ref{clo}.
Consider the set
$$
\triangle _{k,n}=\left \{ x\in \triangle _{k}:f^{l} (x) \in \xi (x),l\in[n,(1+\frac{1}{m})n] \right \}.
$$
By ergodicity of $\mu$ and Birkhoff's ergodic theorem, $\mu(\triangle _{k,n})\to\mu(\triangle _{k})$ as $n \to \infty$. 
For large $n$, $\mu(\triangle _{k,n})>1-\delta $. A set $F \subset M$ is called \textit{$(n,\varepsilon)$-separated set}, if, for $\forall x,y\in F$, there exists some $k\in \left\{ 0,1,...,n-1\right\}$ such that $d(f^{i}(x),f^{i}(y))>\varepsilon$. 
Let $E_{n} \subset \triangle _{k,n}$ be a maximal $(n,\varepsilon)$-separated set in cardinality. Then $\triangle _{k,n}\subset \underset{x\in E_{n}}{\bigcup }D _{f}(x,n,\varepsilon )$, where $D_{f}( x, n, \varepsilon )= \{ y\in M:\underset{0\le l< n}{\max} d(f^{l}(x),f^{l}(y)\le \varepsilon \} $.
So, for infinitely many $n$:
\begin{align}\notag
h_{\mu }(f)- \frac{2}{m}<\frac{\log \sharp E_{n}}{n}. 
\end{align}
Choose $n_{j}$ from these $n$ with $n_{j}\to \infty$ as $j\to \infty$. 
The following inequality holds:
\begin{align}\notag
h_{\mu }(f)-\frac{2}{m}<\frac{\log \sharp E_{n_{j}}}{n_{j}}. 
\end{align}
For $i\in[n_{j},(1+\frac{1}{m})n_{j}]$, let $V_{i}=\left \{x\in E_{n_{j}}:f^{i}x\in \xi(x)\right \}$. 
Let $l$ be the index $i$ maximizing $\sharp V_i$.
Since $e^{\frac{n_{j}}{m}}>\frac{n_{j}}{m}$, we have 
\begin{align} \label{i16}
    \sharp V_{l} \ge e^{n_{j}(h_{\mu}(f)-\frac{3}{m})}.
\end{align}
Note that $R(y,t)=\Psi _{y}(B_{\frac{t\varepsilon_{k}}{2}}^{s}\times B_{\frac{t\varepsilon_{k}}{2}}^{u})\supseteq B(y,\frac{t\vep_k^2}{2})$ for any $y \in  V_{l}$, thus  $\triangle _{k}\subset R(y,t)$.  By Theorem S.4.16 of \cite{KH}, fix  $y\in V_l$, for any $v\in V_l$, the connected component $CC(R(y,t)\cap f^l(R(y,t), f^l(v))$ containing $f^l(v)$ is an admissible $u$-rectangle and $f^{-l}(CC(R(y,t)\cap f^l(R(y,t), f^l(v)))$ is an admissible $s$-rectangle.  Let $t<\vep/6$, then by hyperbolicity, 
$$\diam (f^i(CC(R(y,t)\cap f^{-l}(R(y,t)), v)))\le 3t \max\{e^{-(\lambda-2\epsilon)i},\,e^{-(\lambda-2\epsilon)(l-i)}\}<\frac{\vep}{2},\quad 0\le i\le l. $$ Thus different $v$ correspond to the different $s$-rectangle $CC(R(y,t)\cap f^{-l}(R(y,t)), v))$, which implies that there are $\sharp V_l$ $s$-rectangles. 
Let
$$\Upsilon_0=\bigcap_{n\in\mathbb{Z}} f^{nl}(\bigcup_{v\in V_l} CC(R(y,t)\cap f^{-l}(R(y,t)), v)),$$
then $\Upsilon_0$ is a mixing basic set of $f^l$. Thus $$\Upsilon=\bigcup_{0\le i<l} f^i(\Upsilon_0)$$ is  a  hyperbolic horseshoe  of $f$  with base $\Upsilon_0 \subset R(y,t)$.
By inequality (\ref{i16}), 
\begin{align}\notag
    h_{top}(f|\Upsilon )\ge  \frac{h_{\mu}(f)-\frac3m}{1+\frac1m}.
\end{align}
By letting $m$ sufficiently large, for $\vartheta>0$:
\begin{align}\notag
   h_{top}(f|\Upsilon )\ge h_{\mu}(f)-\vartheta,
\end{align}
which implies 
\begin{align}\label{i11}
   h_{top}(f^l|\Upsilon_0 )= h_{top}(f^l|\Upsilon )=lh_{top}(f|\Upsilon ) \ge lh_{\mu}(f)-l\vartheta. 
\end{align}

We {\bf claim}  that for any periodic points $p\overset{h}{\sim }\mu$ and $q \in \Upsilon_0$, we have $p \overset{h}{\sim } q$.
\begin{proof} [Proof of claim]
Since 
$$
d(y,f^{l}(y))<\diam \xi <\beta(k,\varepsilon), 
$$
by the Closing lemma \ref{clo}, there exists a periodic point $z(y)$ such that 
\begin{itemize}
\item[(i)] $f^l(z(y))=z(y)$;
    \item[(ii)]
 $d(f^{i}z(y),f^{i}y)<\varepsilon$, for $i=0,1,\cdots, l-1$;
\item[(iii)] The local stable manifold (unstable manifold) of  $z(y)$ is an admissible $(s,1)$-manifold ($(u,1)$-manifold) near $y$.
\end{itemize}

For any $x\in \triangle_{k}$, $d(x,y) <\frac{t}{4}\vep_k^2<\frac{1}{4}\vep_k^2$, then $d(x,z(y))<\frac{1}{2}\vep_k^2$, thus $\|\Phi_x^{-1}(x)- \Phi_x^{-1}(z(y))\|\le \vep_k/2$. 
 By item (iii), since $\gamma<1/2$, the local stable (unstable) manifold of $z(y)$ intersects the local unstable (stable) manifold of $x$ transversely.
By ergodicity of $\mu$, then
$z(y) \overset{h}{\sim }\mu$.
For any periodic point $q\in \Upsilon_0 \subset \triangle_{k}$, $z(y)\overset{h}{\sim }q$.
By transitivity and $p\overset{h}{\sim }\mu$, we have $p\overset{h}{\sim }q$.
\end{proof}

\noindent \textit{\textbf{Step 2.} The construction of the new transverse homoclinic point.}

Without loss of generality, suppose that $p$ and $q$  are fixed points of $f^l$. Denote $g=f^l$. 
From equality (\ref{hyp}), we have 
$$
h_{top}(g|\Upsilon_0 )= \underset{n\to \infty }{\lim} \frac{1}{n} \log \sharp (H_g(q,n,\varepsilon)\cap \Upsilon_0). 
$$
To prove Theorem \ref{theorem a}, it suffices to show that for small $\vep'>0$,
\begin{align}\notag
    \sharp (H_g(q,n,\varepsilon )\cap \Upsilon_0) \le \sharp H_g(p,n,\vep'),\quad \forall \,\,n.
\end{align}

Now for small $\vep' >0 $, we can choose two discs $B^{s} = W_{\vep'}^{s}(p)\setminus gW_{\vep'}^{s}(p)$ , $B^{u}= W_{\vep'}^{u}(p)\setminus g^{-1}W_{\vep'}^{u}(p)$. 
It holds that 
$$ \bigcup_{i=0}^{+\infty} g^{-i}(B^{u}) =W_{\vep'}^{u}(p)\setminus \{p\},$$
$$ \bigcup_{i=0}^{+\infty} g^{i}(B^{s}) =W_{\vep'}^{s}(p)\setminus \{p\},$$
By uniform hyperbolicity of $\Upsilon_{0}$, the size of the regular neighborhood $R(y,t)$ is constant for any $y\in \Upsilon_{0}$.
For $y\in  H_g(q,n,\varepsilon )\cap \Upsilon_0$, we select $j>0$ such that $B^{u}$ and $B^{s}$ satisfy that
\begin{itemize}
    \item 
[(i)]
$g^{j}B^{u}\cap R(y,t)$  is an admissible $(u,t)$-manifold near $y$ 
 \item 
[(ii)]$g^{-j}B^{s}\cap R(g^{n} y,t)$ is an admissible $(s,t)$-manifold near $g^{n} y$.
\end{itemize}

Note that   $B^u \cap W^{s}(q)\neq \emptyset$. 
By the $\lambda$-lemma, choose $j_{1}>0$ so that for any $j\ge j_1$,   $g^{j}B^{u}\cap R(y,t)$ contains  an admissible $(u,t)$-manifold $D^u$  near $y$, and close to $W^u_{\vep_k}(q)$ in the $C^{1}$ sense.
 Similarly, since $B^s \cap W^u(q) \ne \emptyset$, there is  $j_{2}>0$ so that for any $j\ge j_2$, $ g^{j}B^{s} \cap R(g^n y,t)$ contains  an admissible $(s,t)$-manifold $D^s$ near $g^{n}(y)$, and  close to $W^s_{\vep_k}(q)$ in the $C^{1}$ sense.
We take the number $j\ge \max \{j_1, j_2 \}$.

We further require how close $D^u$  ($D^s$) should be to $W^u_{\vep_k}(q)$ ($W^s_{\vep_k}(q)$).
For any $y\in H_g(q,n,\varepsilon )\cap \Upsilon_0$,
take $0<a<\frac{\varepsilon}{2}$.
Take a sufficiently large number $j$ as above such that $W^{s}_{a K\vep_k}(y)\cap D^{u} \ne \emptyset$ and  $W^{u}_{aK\vep_k}(g^{n} y)\cap D^{s} \ne \emptyset$. 
\begin{figure}[h]
    \centering
    \includegraphics[width=0.4\linewidth]{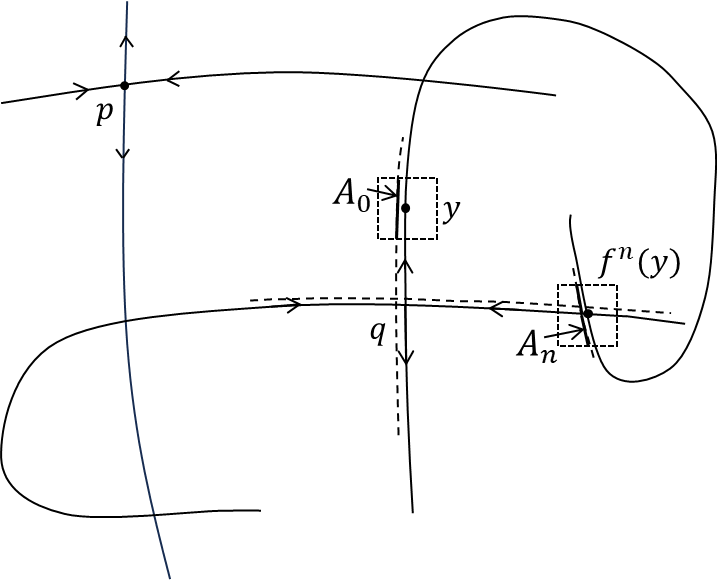}
    \caption{The construction of the new homoclinic points.}
    \label{construction}
\end{figure}

Now we aim to construct a new homoclinic point.
For  $y\in H_g(q,n,\varepsilon)\cap \Upsilon_{0}$, write $A_{0}=D^{u}$.
Define
\begin{align}
A_{1}={}&g(A_{0})\cap R(g(y),t)\notag \\
A_{2}={}&g(A_{1})\cap R(g^{2}(y),t)\notag
\end{align}
$$
 \vdots 
$$
$$
A_{n}=g(A_{n-1})\cap R(g^{n}(y),t).
$$
Then $A_{i}$ is an admissible $(u,t)$-manifold near $g^{i}(y)$ with $i=0,1,\cdots,n$. 
Thus, $A_{n}$ intersects $D^{s}$ transversely at a point $x_{y}$ with $x_{y}\in \Gamma(g,p)$.
It follows that  $d(g^{i}(y),g^{i}(g^{-n}(x_{y})))<aK\vep_k (K^{-1}\vep_k^{-1})=a$,  $\forall \, i=0,1,\cdots, n$.
Then:
\begin{align}
\Theta ^{s} (g^{-(n+j)}(x_{y}),p,\vep' )&{}=2j+ n,\notag\\
\Theta ^{u} (g^{-(n+j)}(x_{y}),p, \vep' )&{}=0.\notag   
\end{align}
So
$$
g^{-(n+j)}(x_{y})\in H_g(p,2j+n,\vep'),
$$ 
and $j$ is independent of $n$. 

If $y_{1}\ne y_{2}\in H_g(q,n,\varepsilon)\cap \Upsilon_{0}$, by the expansiveness of $g|\Upsilon_{0}  $, then for $\vep$ be smaller than the expansiveness constant,  $d(g^{i}(y_{1}),g^{i}(y_{2}))>\varepsilon$, for some $ i=0,1,\cdots,n$.
Then
\begin{align}\notag
    \varepsilon &<d(g^{i}(y_{1}),g^{i}(y_{2}))\\ \notag
    &\le d(g^{i}(y_{1}),g^{i}(x_{y_{1}}))+d(g^{i}(x_{y_{1}}),g^{i}(x_{y_{2}}))+d(g^{i}(x_{y_{2}}),g^{i}(y_{2}))\\ \notag
    &\le 2a+d(g^{i}(x_{y_{1}}),g^{i}(x_{y_{2}})).
\end{align}
Since $2a<\varepsilon$, we have $x_{y_{1}}\ne x_{y_{2}}$.

Therefore:
\begin{align} \label{i15}
    \sharp (H_g(q,n,\varepsilon)\cap \Upsilon_0 )\le \sharp H_g(p,2j+n,\vep'),\quad \forall \, n .
\end{align}

\noindent \textit{\textbf{Step 3.} Estimation on metric entropy and homoclinic growth.}

From (\ref{hyp}), (\ref{i11}) and (\ref{i15}), 
\begin{align}
    lh_{\mu}(f)-l\vartheta &{}<h_{top}(g|\Upsilon _0)\notag\\
                &{}=\underset{n\to \infty}{\lim} \frac{1}{n} \log \sharp( H_g(q,n,\varepsilon)\cap \Upsilon_0) \notag\\ \notag
                &{}\le \underset{n\to \infty }{\limsup} \frac{1}{n} \log \sharp H_g(p,2j+n,\vep') \\ \notag
               &\le  \underset{n\to \infty }{\limsup} \frac{1}{n} \log \sharp H_f( p, (2j+n)l,\vep')\\ \notag
                &=  l\, \underset{n\to \infty }{\limsup} \frac{1}{(2j+n)l} \log \sharp H_f( p, (2j+n)l,\vep')
                \\ \notag
                &\le   l\, \underset{n\to \infty }{\limsup} \frac{1}{n} \log \sharp H_f(p, n,\vep').
\end{align}
Thus Theorem \ref{theorem a} follows by dividing $l$ on both sides and the arbitrariness of $\vartheta$.
\end{proof}

\section{exp-type growth}\label{exp}
Having established the log-type, we now turn to the exp-type case, which requires a different approach via symbolic dynamics.

\subsection{Symbolic dynamics and homoclinic point}\label{exp1}
The following is based on \cite{S} and \cite{SBO}. 
Let $f:M\to M$ be a $C^{r}$ $(r>1)$ diffeomorphism with $h_{\mu}(f)>0$ on a compact Riemannian manifold $M$, where $\mu$ is an ergodic invariant probability measure.
Sarig \cite{S}  constructed a countable Markov partition $\mathcal{R}$ of $M$ and developed symbolic dynamics for surface diffeomorphisms, and it is also generalized to higher dimension by  Ben Ovadia \cite{SBO}.
Let $G$ be the directed graph with vertices $\mathcal{V}:=\mathcal{R}$ and edges: 
$$\mathcal{E} :=\left \{ (R_{1},R_{2})\in \mathcal{V} ^{2}:\,\,R_{1},R_{2}\in \mathcal{V}\,\, s.t.\,\, R_{1} \cup f^{-1}(R_{2}) \ne \emptyset \right\}.
$$
If $(R_{1},R_{2})\in \mathcal{E}$, write $R_{1}\to R_{2}$.
The topological Markov shift is:
$$
\Sigma :=\Sigma (G)=\left \{ (R_{i})_{i\in \mathbb{Z} }\in \mathcal{R} ^{ \mathbb{Z}}:R_{i}\to R_{i+1} \text{ for all }i\in \mathbb{Z} \right \}.
$$
The dynamical system $\sigma:\Sigma \to \Sigma$ is defined on the metric space $(\Sigma,d)$, where
$$d(\underline{v},\underline{u}):=\exp [-\inf \left \{ |i| : v_{i}\ne u_{i}\right \} ],
$$
and the left shift map on $\Sigma $ by 
$$\sigma: (v_{i})_{i\in \mathbb{Z}}\to (v_{i+1})_{i\in \mathbb{Z}}.$$ 
Since every vertex in $G$ constructed in \cite{S, SBO} has finite degree, $\Sigma$ is locally compact.
And this $G$ is \textit{irreducible}.
That is to say that for any two vertices $v$, $u$ from $G$ there exists a path from $v$ to $u$ on the graph $G$. 
The regular part of $\Sigma$ is :
$$
\Hat{\Sigma}:=\left \{ (v_{i})\in \Sigma : \exists\,\, v,w\in\mathcal{V}, \exists\,\, n_{k},m_{k}\uparrow \infty \,\,s.t.\,\,v_{n_{k}}=v \text{ and }v_{m_{k}}=w \right \},
$$
which has full measure for every $\sigma$-invariant probability measure on $\Sigma$.

Let $\kappa(\mu):=\min \left \{ \left | \kappa _{i}  \right | :\kappa _{i} \text{ is a Lyapunov exponent of }\mu \right \}.$ 
If $0<\kappa<\kappa(\mu)$, $\mu$ is called \textit{$\kappa$-hyperbolic}.

\begin{theorem}[\cite{S, SBO}] \label{the 5}
Let $f$ be a $C^{r}$ $(r>1)$ diffeomorphism on a compact manifold $M$. For every $\kappa>0$, there are a locally compact countable state Markov shift $\Sigma$ and a H\"older continuous map $\pi:\Sigma \to M $ such that $\pi \circ \sigma =f\circ \pi $ and 
\item[(1)] $\pi:\Hat{\Sigma} \to M $ is finite-to-one, More precisely, if $x=\pi(\underline{x})$ where $x_{i}=a$ for infinitely many $i<0$ and $x_{i}=b$ for infinitely many $i>0$, then $\left | \pi^{-1}(x) \right |$ is bounded by a constant $C(a,b)$.
\item[(2)] $v(\pi(\Hat{\Sigma}))=1$ for every $\kappa$-hyperbolic measure $v\in \mathcal{M}^{*}(f)$. Moreover, there exists an ergodic measure $\hat{v}$ on $\Sigma$ such that $\pi_{*}(\hat{v})=v$. Conversely, if $\hat{v}$ is an ergodic measure on $\Sigma$, then $\pi_{*}(\hat{v})$ is $f$-ergodic hyperbolic and $h_{\hat{v}}(\sigma)=h_{v}(f)$.
\item[(3)] For any $x\in \pi(\Hat{\Sigma})$, there is a splitting $T_{x}M=E^{s}(x)\oplus E^{u}(x) $, where 

$(i)\underset{n\to+\infty }{\limsup} \frac{1}{n}\log\left \| Df_{x} ^{n}|_{E^{s}(x)}\right \| \le -\frac{\kappa}{2}$;

$(ii)\underset{n\to+\infty }{\limsup} \frac{1}{n}\log\left \| Df_{x} ^{-n}|_{E^{u}(x)}\right \| \le -\frac{\kappa}{2}$.
\item[(4)] For every $\underline{u}
\in \Sigma$, there are two $C^{1}$ sub-manifolds $V^{u}(\underline{u})$, $V^{s}(\underline{u})$ pass through $x:=\pi(\underline{u})$ such that

$(i) \forall y\in V^{s}(\underline{u}),\forall n\ge0,d(f^{n}(y),f^{n}(x))\le e^{-\frac{n\kappa }{2} }$ and $T_{x}V^{s}(\underline{u})=E^{s}(x)$,

$(ii) \forall y\in V^{u}(\underline{u}),\forall n\ge0,d(f^{-n}(y),f^{-n}(x))\le e^{-\frac{n\kappa }{2} }$ and $T_{x}V^{u}(\underline{u})=E^{u}(x)$.
\end{theorem}

For a periodic point $\underline{v} \in \Sigma$ of period $m$, i.e. $\sigma^{m}\underline{v}=\underline{v}$,  define:
\begin{eqnarray*}
    W^{s}(\underline{v})&:=&\left \{ \underline{u}:\exists \,\, N,  v_{i}= u_{i}, \text{ for all }i\ge N \right \},\\
    W^{u}(\underline{v})&:=&\left \{ \underline{u}:\exists \,\, N,  v_{i}=u_{i}, \text{ for all }i\le -N \right \}. 
\end{eqnarray*}
The stable set $W^{s}(\underline{v})\subset \Sigma$ and the unstable set $ W^{u}(\underline{v}) \subset \Sigma$ act as the stable manifold and unstable manifold respectively. 
Therefore, we can get the following result:
\begin{prop} \label{prop 3}
    Let $p\in M$ be a hyperbolic periodic point with $\pi(\underline{v})=p$. If $\underline{u} \in W^{u}(\underline{v}) \cap W^{s}(\underline{v})$, then $z=\pi (\underline{u})$ is a transversely homoclinic point of $p$.
\end{prop}
\begin{proof}[Proof of Proposition \ref{prop 3}]
By the definition of $W^{s}(\underline{v})$,
 $d(\sigma^{n}(\underline{u}),\sigma^{n}(\underline{v}))\le e^{-|N-n|}$.
Denote $z=\pi(\underline{u})$.
By Theorem \ref{the 5}, the map $\pi:\Sigma \to M$ is a H\"older continuous, then 
$d(f^{n}(z),f^{n}(p))\to 0$ exponentially fast.
So $z\in W^{s}(p)$.
Similarly, $z\in W^{u}(p)$.
Then $z\in W^{s}(p)\cap W^{u}(p)$.

Now we need to prove that the intersection at point $z$ is transverse, i.e. $T_{z}M=T_{z}W^{s}(p)\oplus T_{z}W^{u}(p)$.
By $(4)$ of Theorem \ref{the 5}, there are two $C^{1}$ sub-manifolds $V^{u}(\underline{u})$, $V^{s}(\underline{u})$ pass through $z$ such that $V^{u}(\underline{u}) \subset W^{u}(p)$ and $V^{s}(\underline{u}) \subset W^{s}(p)$.
So $T_{z}W^{s/u}(p)=T_{z}V^{s/u}(\underline{u} )=E^{s/u}(z)$ with $E^{s/u}(z)$ as in (3) of Theorem \ref{the 5}.
Since $E^{s}(z) \cap E^{u}(z)=\emptyset$, the manifolds $W^{u}(p)$, $W^{s}(p)$ are transversal at the point $z$.
\end{proof}

\subsection{Proof of Theorem \ref{theorem b}}
\begin{proof} 
Let $\underline{p}\in \Sigma$ be a periodic point with period $m$ and $\pi(\underline{p})=p$.
Fix a vertex $w_{0} \in \mathcal{V}$.
Let $\underline{q}\in \Sigma$ be a periodic point with period $l$, $q_{0}= w_{0}$ and $\pi(\underline{q})=q$ .
And the point $q$ satisfies $p\ne q$.
Define the set 
$P_{n}(w_{0})=\left \{ \underline{v}\in \Sigma :\sigma ^{n}(\underline{v})=\underline{v},\,v_{0}=w_{0} \right \} $.
Our proof will be divided into three steps in the following.

\textit{\textbf{Step 1.} Construct two transverse intersection points that belong to $W^{s}(p)\cap W^{u}(q)$ and $W^{u}(p)\cap W^{s}(q)$ respectively.}

Since $G(\mathcal{V})$ is irreducible, there are two paths from $p_{m-1}$ to $w_{0}$,  and from $w_{0}$ to $p_{0}$.
Exactly, there exist two sequences of symbols $p_{m-1}\to r_{K} \to \cdots \to r_{1} \to w_{0}$ and $w_{0}\to u_{1} \to \cdots \to u_{N-2} \to p_{0}$.
Construct 
$$
\Gamma_{1} :=\overbrace{q_{0},\cdots ,q_{l-1}}^{\infty},w_{0},u_{1}, \cdots ,u_{N-2},\overbrace{p_{0},\cdots ,p_{m-1}}^{\tilde{K}},p_{0}\mid_0,\cdots ,p_{m-1},\overbrace{p_{0},\cdots ,p_{m-1}}^{\infty}
$$
satisfying that there are $\tilde{K}$-copies of $p_{0}, \cdots ,p_{m-1}$ in the left of 0 position. 
Then for all $i$,
$d(\sigma^{i}\Gamma_{1},\sigma^{i}\underline{p})\le e^{-i-\tilde{K}m}$ and
$d(\sigma^{-i}\sigma^{-(N+\tilde{K}m)}(\Gamma_{1}),\sigma^{-i}\underline{q})\le e^{-i}$.
Since $\pi:\Sigma\to M$ is H\"older continuous, there exist $C>0$, $\delta>0$  such that for any $\underline{v}_{1},\underline{v}_{2} \in \Sigma$,
$$
d(\pi(\underline{v}_{1}),\pi(\underline{v}_{2}))\le Cd(\underline{v}_{1},\underline{v}_{2})^{\delta}.
$$
Then
$d(f^{i}\pi(\Gamma_{1}),f^{i}\pi(\underline{p})),d(f^{-i-N-\tilde{K}m}\pi(\Gamma_{1}),f^{-i}\pi(\underline{q}))\to 0$ exponentially fast as $i\to \infty$.
So $\pi (\Gamma_{1}) \in W^{s}(p)$ and $f^{-(N+\tilde{K}m)}\pi( \Gamma_{1}) \in W^{u}(q)$.
Since $p\ne q$, $\pi(\Gamma_{1})$ is a transverse intersection and $d(\pi(\Gamma_{1}),p)>0$.
The point $\pi(\Gamma_{1})$ belongs to some local stable manifold of $p$.
Now we can calculate the size of this local stable manifold.  
It holds that  for $i\ge 0$,
$$d(f^i\pi(\Gamma_1), f^i(p))= d(f^i(\pi(\Gamma_1), f^i(\pi(\underline{p})))\le Ce^{-(i+\tilde{K}m)\delta}\le Ce^{-\tilde{K}m\delta}, $$
thus $\pi(\Gamma_1)\in W^s_{Ce^{-\tilde{K}m\delta}} (p)$.  
Denote  $\eta=Ce^{-\tilde{K}m\delta}$. 
Assume that $d(\pi(\Gamma_{1}),p)= \eta_{1}$.

Construct 
$$
\Gamma_{2} :=\overbrace{p_{0},\cdots ,p_{m-1}}^{\infty},p_{0}\mid_0,\cdots ,p_{m-1},\overbrace{p_{0},\cdots ,p_{m-1}}^{\tilde{K}-1},r_{K}, \cdots ,r_{1},\overbrace{q_{0},\cdots ,q_{l-1}}^{\infty}
$$
satisfying  there are $(\tilde{K}-1)$-copies of $p_{0}, \cdots ,p_{m-1}$ in the right of 0 position. 
Similarly, we get $\pi (\Gamma_{2}) \in W^{u}(p)$ and $f^{K+\tilde{K}m}\pi( \Gamma_{2}) \in W^{s}(q)$.
So $\pi(\Gamma_{2})\in W^{u}(p)$ is a transverse intersection and $d(\pi(\Gamma_{2}),p)>0$.
And $\pi(\Gamma_{2})\in W^{u}_{\eta}(p)$.
Assume that $d(\pi(\Gamma_{2}),p)= \eta_{2}$.

 \textit{ \textbf{Step 2.}For any sequence $w_{0},w_{1}, \cdots ,w_{n-1}$ from $P_{n}(w_{0})$, we construct $\gamma \in \Sigma$ that contains the sequence $w_{0},w_{1}, \cdots ,w_{n-1}$ and satisfies that, by iterating the point $\pi(\gamma)$, there is a transverse homoclinic point of $p$ belonging to $W^{u}_{\varepsilon}(p)\setminus f^{-1}W^{u}_{\varepsilon}(p)$.
Then we will calculate the time from $W^{u}_{\varepsilon}(p)\setminus f^{-1}W^{u}_{\varepsilon}(p)$ to $W^{s}_{\varepsilon}(p)\setminus fW^{s}_{\varepsilon}(p)$. 
}

For any $0<\varepsilon\le \frac{1}{100}\min \left \{ \eta _{1},\eta _{2}  \right \} $,  we claim that there exist positive integers $J_{1}$ and $J_{2}$  such that $f^{J_{1}}\pi(\Gamma_{1})\in W^{s}_{\varepsilon}(p)\setminus fW^{s}_{\varepsilon}(p)$ and $f^{-J_{2}}\pi(\Gamma_{1})\in W^{u}_{\varepsilon}(p)\setminus f^{-1}W^{u}_{\varepsilon}(p)$.
We will only prove the former result, and the latter result can be proved in a similar way.
In deed, there exist  $\lambda\in (0,1)$ and $C_p>0$ such that  for any $y,z\in W^{s}_{\eta}(p)$, 
$$d(f^i(y), f^i(z))\le C_p\lambda^id(y, z),\,\,\forall\, i\ge 0.
$$
For $\pi(\Gamma_{1})$ and $p$,
$$
d(f^i \pi(\Gamma_{1}), f^i p)\le C_p\lambda^id(\pi(\Gamma_{1}), p) \le C_p\lambda^i \eta_{1}.
$$
It is only necessary to meet 
$$
C_p\lambda^i \eta_{1} \le \varepsilon,\, i.e. \,\, i\ge \frac{\ln(\varepsilon /C_{p}\eta_{1} )}{\ln\lambda }.
$$
Take $J_{1}=\left [  \frac{\ln(\varepsilon /C_{p}\eta_{1} )}{\ln\lambda } \right ] +1$.
So $f^{J_{1}}\pi(\Gamma_{1})\in W^{s}_{\varepsilon}(p)\setminus fW^{s}_{\varepsilon}(p)$.
Take $B(\pi(\Gamma_{1}),\tilde{\rho_{1}})\subset W^{s}_{\varepsilon}(p)\setminus fW^{s}_{\varepsilon}(p)$ and $B(\pi(\Gamma_{2}),\tilde{\rho_{2}})\subset W^{u}_{\varepsilon}(p)\setminus f^{-1}W^{u}_{\varepsilon}(p)$.
Now we hope to construct $\gamma \in \Sigma$ such that 
$$d(f^{J_{1}}\pi (\gamma),f^{J_{1}}\pi (\Gamma_{1}))\le C_{p}\lambda^{J_{1}}d (\pi(\gamma),\pi(\Gamma_{1}))\le \tilde{\rho_{1}},\,\,i.e.\,\, d (\pi(\gamma),\pi(\Gamma_{1}))\le \frac{\tilde{\rho_{1}}}{C_{p}\lambda^{J_{1}}},$$
and $d (\pi(\gamma),\pi(\Gamma_{2}))\le \frac{\tilde{\rho_{2}} }{C_{p}\lambda^{J_{2}}}$.
Denote $\rho_{1}=\frac{\tilde{\rho_{1}} }{C_{p}\lambda^{J_{1}}}$ and $\rho_{2}=\frac{\tilde{\rho_{2}}}{C_{p}\lambda^{J_{2}}}$.
Take $L$ to satisfy that $Ce^{-(Lm+\tilde{K}m+N)\delta}<\rho_{1}$ and $Ce^{-(Lm+\tilde{K}m+K)\delta}<\rho_{2}$.
For any sequence $w_{0},w_{1}, \cdots ,w_{n-1}$ from $P_{n}(w_{0})$, construct a $\gamma \in \Sigma$:
\begin{align}
    \overbrace{p_{0},\cdots ,p_{m-1}}^{\infty},\overbrace{p_{0},\cdots ,p_{m-1}}^{\tilde{K}},r_{K}, \cdots ,r_{1},\overbrace{q_{0},\cdots ,q_{l-1}}^{L},w_{0},w_{1}, \cdots ,w_{n-1},\\ \notag \overbrace{q_{0},\cdots ,q_{l-1}}^{L},w_{0},u_{1}, \cdots ,u_{N-2},\overbrace{p_{0},\cdots ,p_{m-1}}^{\tilde{K}},p_{0}
    \mid_0,\cdots ,p_{m-1},\overbrace{p_{0},\cdots ,p_{m-1}}^{\infty}
\end{align}
satisfies there are $\tilde{K}$-copies of $p_{0}, \cdots ,p_{m-1}$ in the left of 0 position.

For any $\gamma$ as above, $\pi(\gamma)$ is a transversely homoclinic point of $p$.
And $f^{J_{1}}\pi(\gamma) \in W^{s}_{\varepsilon}(p)\setminus fW^{s}_{\varepsilon}(p)$, $f^{-(J_{2}+2\tilde{K}m+2Lm+K+N+n)}\pi(\gamma) \in W^{u}_{\varepsilon}(p)\setminus f^{-1}W^{u}_{\varepsilon}(p)$.
So 
$$f^{-(J_{2}+2\tilde{K}m+2Lm+K+N+n)}\pi(\gamma) \in H(p,J_{1}+J_{2}+2\tilde{K}m+2Lm+K+N+n,\varepsilon).
$$
Denote $\tilde{N}=J_{1}+J_{2}+2\tilde{K}m+2Lm+K+N$.
$\tilde{N}$ do not depend on $n$. 
For any $\gamma \in \Sigma$, it is constructed as above and $\pi(\gamma) \in M$.
By construction, there are infinitely many $i<0$ such that $\gamma_{i}=p_{0}$ and infinitely many $i>0$ such that $\gamma_{i}=p_{0}$.
By Theorem \ref{the 5}, there are at most $C(p_{0},p_{0})$ $\pi$-pre-images of $\pi(\gamma)$.

Fix a number $n$, then
$$
\frac{1}{C(p_{0},p_{0})}\sharp P_{n}(w_{0})\le \sharp H(p,\tilde{N}+n,\varepsilon).
$$

\textit{\textbf{ Step 3.}  The metric entropy and the growth of the homoclinic points.}

From \cite{Gu1} and \cite{Gu2}, a topological transitive topological Markov shift $\Sigma$ adimits at most one measure of maximal entropy, and that such a measure exists iff $\exists q\in\mathbb{N} $ such that for every vertex $w_{0}$ in $G$ and some $C_{0}>1$,
$$
C_{0}^{-1}<\sharp P_{n}(w_{0})e^{-nh_{\max}(\Sigma)}<C_{0}
$$
as $n\to \infty$ and $n\in q\mathbb{N}$, where 
$$
h_{\max}(\Sigma)=\sup\left \{ h_{\mu}(\sigma ):\mu \text{ a }\sigma-\text{invariant Borel probability measure on }\Sigma  \right \}.
$$
So
$$
\frac{1}{C_{0}C(p_{0},p_{0})}<\frac{1}{C(p_{0},p_{0})}\sharp P_{n}(w_{0})e^{-nh_{\max}(\Sigma)} \le \sharp H(p,\tilde{N}+n,\varepsilon) e^{-nh_{\max}(\Sigma)}.
$$
By (1) of Theorem \ref{the 5}, the mapping $\pi$ is bound-to-one.
Thus, the inequality implies
$$
\underset{n\to\infty, q|n }{\liminf}\sharp H(p,n,\varepsilon) e^{-nh_{\max}(\Sigma)}>0.
$$
Let $\widehat{\mu}$ be a measure of maximal entropy, by (2) of Theorem \ref{the 5}, then
$$h_{\max}(\Sigma)=h_{\widehat{\mu} }(\sigma)=h_{\mu }(f)=\max\left \{ h_{\nu }(f): \nu \text{ a }f-\text{invariant measure on }M \right \}.$$
So
$$
\underset{n\to\infty, q|n }{\liminf}\sharp H(p,n,\varepsilon) e^{-nh_{\mu}(f)}>0.
$$
as $n\to \infty$ and $n\in q \mathbb{N}$.
We get the result.
Moreover, by the variational principle, this implies
$$
\underset{n\to\infty, q|n }{\liminf}\sharp H(p,n,\varepsilon) e^{-nh_{top}(f)}>0.
$$
\end{proof}

\begin{proof}[Proof of Corollary 1]
    The periodic point $p\in M$ is the point satisfying that $\pi(\underline{p})=p$, where the point $\underline{p}$ is a periodic point on $ \hat{\Sigma}$.
    Because $C^{\infty}$ diffeomorphisms on the compact manifold have measures of maximal entropy,
then by Theorem \ref{theorem b},
$$
\underset{n\to\infty }{\limsup}\frac{\sharp H(p,n,\varepsilon)}{e^{nh_{top}(f)}} >0.
$$
\end{proof}

\section{Super-exponential growth}

In contrast to the lower-bound estimation of the growth rate for transverse homoclinic points,
we now demonstrate that in Newhouse domain the growth rate of homoclinic points can be super-exponential.

Assume that $f,g \in \Diff^{r}(M)$.
Define the $C^{r}$-distance between $f$ and $g$:
$$
d_{r}(f,g)=\underset{0\le k\le r}{\max} \,\,\underset{x\in M}{\sup}\left \| D^{k} f(x)-D^{k} g(x) \right \|,
$$
where $1\le r < \infty$.
And the $C^{\infty}$-distance is  defined as 
\begin{align}\label{distance}
    d_{\infty}(f,g)=\sum_{r=1}^{\infty }  \frac{1}{(1+r)^{2}} \frac{d_{r}(f,g)}{1+d_{r}(f,g)}.
\end{align}
For a $\delta>0$, two $C^{r}$-smooth maps are $\delta$-close if a $C^{r}$-distance between them does not exceed $\delta$, 
and write $f\longmapsto _{\delta ,r}g$ if $g$ is a $C^{r}$ diffeomorphism $\delta$-close to $f$ with respect to $d_{r}$.

A disk in $M$ is a $C^r$ embedding $\phi: B^s \times B^u  \to M$ where $s+u=\dim M$.
We will call a disk $(D,\phi)$ in  $M$ where $D$ is the set $\phi (B^s \times B^u)$.
Let $(D^u,\phi^u)$ and $(D^s,\phi^s)$ be a $C^2$ $u$-disk and a $C^2$ $s$-disk in $M$ with $s+u=\dim M$.
We say that $D^s$ and $D^u$ have a \textit{non-degenerate tangency} at $z$ if there are $C^2$ coordinates $(x,y)=(x_1,\cdots,x_s,y_1,\cdots,y_u)$ near $z$ with $D^s=\{(x,y):y=0\}$ and a curve $t\longmapsto \gamma(t)$ for $t$ in an interval $I$ about 0 such that:
\begin{itemize}
    \item [(i)]$\gamma(0)=z;$
    \item [(ii)] $\gamma(t)\in D^u$ for $t\in I$;
    \item [(iii)] $\gamma'(0)\ne 0$ and $T_zD^s \cap T_zD^u$ is the one-dimensional subspace of $T_zM$ spanned by $\gamma'(0)$;
    \item [(iv)]  $\gamma''(0) \ne 0$ and $\gamma''(0)\notin T_zD^s \cap T_zD^u$.
\end{itemize}
Here $\gamma'(0)$ and $\gamma''(0)$ are the first and second derivatives of $\gamma$ at 0.

Let $\Lambda$ be a hyperbolic basic set for a $C^r$ diffeomorphism $f:M\to M$ with $r\ge 2$.
A non-degenerate tangency $z$ of $W^u(x,f)$ and $W^s(y,f)$ for $x,y\in \Lambda(f)$ will be called a \textit{non-degenerate homoclinic tangency} for $\Lambda(f)$.
$\Lambda(f)$ is a \textit{wild hyperbolic set} if each $g$ $C^r$-near $f$ has the property that $\Lambda(g)$ has the property that $\Lambda(g)$ has a non-degenerate homoclinic tangency.

\begin{proof}[Proof of Theorem \ref{theorem c}] 

For $r\ge 2$, let $\mathcal{N}^{r} \subset \Diff ^{r}(M)$ be the $C^r$ Newhouse domain. There is a $C^{r}$ 
 dense set $D_r \subset \mathcal{N}^{r}$ such that every $f_0\in D_r$ has a homoclinic tangency  with respect to some periodic point $p_{f_0}$.  
 Consider an arbitrary sequence of positive integers $\left\{a_{n}\right\}_{n=1}^{\infty}$.

\textit{\textbf{Step 1.} Choose two different periodic orbits and obtain intervals of homoclinic tangencies with respect to one of them.} 

Take a small $\delta >0$. By Theorem 1 of \cite{N2}, there exists an open set $U_{1}$ in the $\frac{1}{4}\delta$-$C^r$ neighborhood of $f_0 \in D_{0}$ such that for any $g_0 \in U_{1}$, there is a wild basic set $\Lambda(g_0)$ containing $p_{g_0}$ that exhibits a tangency between $W^s(x)$ and $W^u(y)$ for some $x, y \in \Lambda(g_0)$, where $p_{g_0}$ denotes the  continuation of $p_{f_0}$.

By shrinking $U_{1}$ if necessary, we may further assume that $\Lambda(g_0)$ contains another periodic point $q_{g_{0}}$ lying on a distinct orbit from that of $p_{g_0}$, and that $q_{g_{0}}$ depends continuously on $g_{0}$. Moreover, by considering an appropriate iteration, we may assume both $p_{g_{0}}$ and $q_{g_{0}}$ are fixed points of $g_0$.

By Lemma 8.4 of \cite{N3}, there is a $C^r$-dense subset $\widehat{U_1}\subset U_{1}$ such that for every $g_0\in \widehat{U_1}$ and $q_{g_0}\in \Lambda(g_0)$, we can get a homoclinic tangency for $q_{g_0}$.

 Moreover, by Proposition 5 and Lemma 3 of \cite{Kal}, there exists a $C^{r}$-perturbation  $g_0\longmapsto _{\frac{1}{4}\delta ,r}\hat{g}$ such that $\hat{g}$ has an interval $\gamma$ of tangencies between $W^{u}(q_{\hat{g}})$ and $W^{s}(q_{\hat{g}})$.

\textit{\textbf{ Step 2.} Choose homoclinic points of $p_{\hat{g}}$ approximating to $q_{\hat{g}}$. By iteration, they enter into the set with large order.}  

Let $\eta=d(p_{\hat{g}},q_{\hat{g}})$.
For small $\hat{\vep}<\eta/2$,  since the homoclinic points of $p_{\hat{g}}$ are dense in $\Lambda(\hat{g})$, there exists a sequence $\{r_{i}\}_{i\ge 1}\subset \Gamma(p_{\hat{g}})$  such that $r_{i}\to q_{\hat{g}}$.
For each $i$, there exists $t_i<0$ such that $\hat{g}^{t_{i}}r_{i}\in W^u_{\hat{\vep}}(p_{\hat{g}})\setminus \hat{g}^{-1}(W^u_{\hat{\vep}}(p_{\hat{g}}))$. 

We claim that $t_i \to -\infty$ as $i\to \infty$.

\begin{proof}[Proof of the Claim]
Let $\eta_i=d(r_i,q_{\hat{g}})$ with $\eta_i \ll \eta $ and $A=\underset{x\in M}{\sup} \| D_x\hat{g}^{-1}\|$.
Then $\eta_i \to 0$ as $i\to \infty$.
For any $ t\ge 1$, 
\begin{align*}
    d(\hat{g}^{-t}(r_i),\hat{g}^{-t}(q_{\hat{g}}))<A^{t}\eta_{i}.
\end{align*}
For any $i$, 
there exists a positive integer $N_i=[\frac{\ln(\frac{\eta}{2}/\eta_i)}{\ln A}]$ such that $$d(\hat{g}^{-t}(r_i), q)<\eta/2,\quad \forall\,\,1\le t\le N_i,$$ which together with  $\hat{\vep}<\eta/2$ implies $$\hat{g}^{-t}(r_i)\notin W^{u}_{\hat{\vep}}(p_{\hat{g}}),\quad \forall \,\,1\le t\le N_i.$$
Thus, $-t_i>N_i$.
 Note that $N_{i}\to \infty$ 
since there is $\eta_i \to 0$.
So $t_i\to -\infty$ as $i\to \infty$.

\end{proof} 

\textit{\textbf{ Step 3.}  Choose small balls entering the $\hat{\vep} $-neighborhood of $p_{\hat{g}}$ at the same time.}

Let $\eta_0>0$ be small enough such that if $V=\{y\in M: \,\, d(y,\Lambda(\hat{g}))\le \eta_0\}$ and $B_{\eta_0}(\gamma)=\{y\in M:\,\,d(y,\gamma)<\eta_0 \}$, then $\underset{n}{\bigcap } f^{n}V=\Lambda (\hat{g})$ and $V\cap B_{\eta_0}(\gamma)=\emptyset$.

Choose $z$ in the interior of $\gamma$ such that
there are integers $\kappa _{1},\kappa _{2}>0$ with: $$\hat{g}^{-\kappa_{1}}(z) \in \Int(W^{u}_{\eta_0}(q_{\hat{g}})),$$ $$\hat{g}^{\kappa_{2}}(z) \in \Int( W^{s}_{\eta_0}(q_{\hat{g}})).$$
Choose $\varepsilon_{0}<\eta_0$ such that $\hat{g}^{j}B_{\varepsilon_{0}}(z)\cap B_{\varepsilon_{0}}(z) =\emptyset$ for $-\kappa_{1}\le j\le \kappa_{2}$ and $j\ne 0$.

Let $\gamma_{us} \subset \gamma\cap B_{\varepsilon_{0}}(z)$ be a small interval containing $z$ such that:

\begin{itemize}\item  $\hat{g}^{-j}\gamma_{us} \cap B_{\varepsilon_{0}}(z)=\emptyset$,
and $\hat{g}^{j}\gamma_{us} \cap B_{\varepsilon_{0}}(z)=\emptyset$ for $j\ge1$.\\
\item $\hat{g}^{-\kappa_1}\gamma_{us}\subset W^{u}_{\eta_0}(q_{\hat{g}})$, $\hat{g}^{\kappa_2}\gamma_{us}\subset W^{s}_{\eta_0}(q_{\hat{g}})$.
\end{itemize}

\textit{\textbf{ Step 4.}} Perturb to create a new  interval of tangencies between $W^{s}(p_{\hat{g}})$ and $W^{u}(p_{\hat{g}})$.
Then perturb further to generate infinitely many 88transverse homoclinic points of $p_{\hat{g}}$.

Since $r_i\to q_{\hat{g}}$, $W^u_{\eta_0}(r_i)$ is $C^r$-close to $W^u_{\eta_0}(q_{\hat{g}})$. 
Choose a segment $I_{r_i}\subset W^u_{\eta_0}(r_i)$ (resp. $J_{r_i}\subset W^s_{\eta_0}(r_i)$ ) which is $C^r$-close to $\hat{g}^{-\kappa_1}(\gamma_{us})$ (resp. $\hat{g}^{\kappa_2}(\gamma_{us})$ ). 
Then
$\hat{g}^{\kappa_1} (I_{r_i})$ and $\hat{g}^{-\kappa_2} (J_{r_i})$ are both $C^r$-close to $\gamma_{us}$, thus $\hat{g}^{\kappa_1} (I_{r_i})$ and $\hat{g}^{-\kappa_2} (J_{r_i})$ are $C^r$-close to each other. 

\begin{figure}[h]
    \centering
    \includegraphics[width=0.6\linewidth]{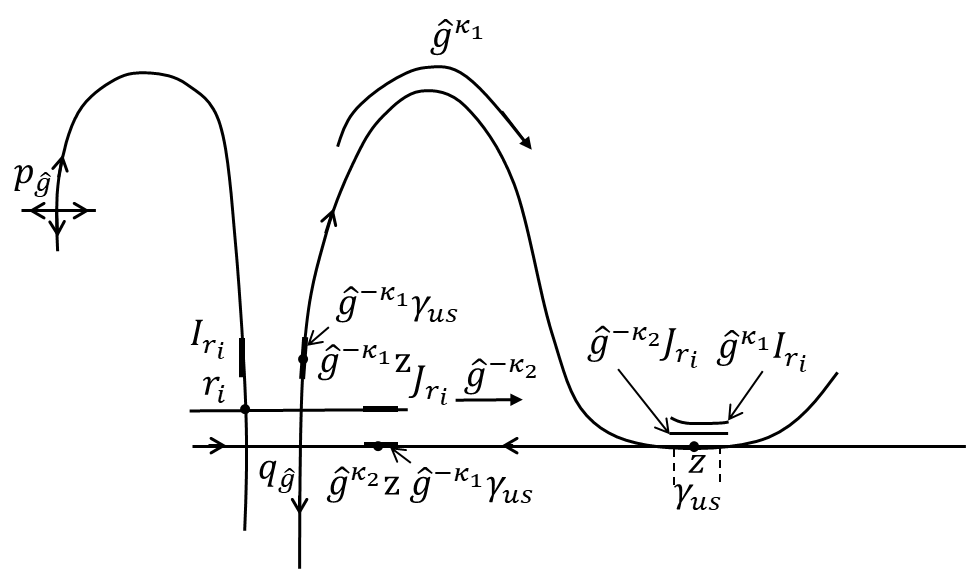}
    \caption{Construct an interval of tangencies between $W^u(p_{\hat{g}})$ and $W^s(p_{\hat{g}})$}
    \label{P9}
\end{figure}

Using a bump function $\phi_1$ with $\phi_{1}(w)=w$ for $w\notin B_{\varepsilon_{0}}(z)$, there exists a $C^{r}$-perturbation $\hat{g}\longmapsto _{\frac{1}{4}\delta ,r}f_1$ such that $f_1$ has a new interval $\gamma_{p_{\hat{g}}}$ (close to $\gamma_{us}$) of homoclinic tangencies between $W^{u}_{loc}(p_{\hat{g}},f_1)$ and $W^{s}_{loc}(p_{\hat{g}},f_1)$.

Now we will choose a new interval in $\gamma_{p_{\hat{g}}}$ where all points have the same order.
For $\gamma_{p_{\hat{g}}}$, there exists $k_i>0$ such that $f_1^{-k_i}\gamma_{p_{\hat{g}}}\cap (W^u_{\hat{\vep}}(p_{\hat{g}}) \setminus f_1^{-1} W^u_{\hat{\vep}}(p_{\hat{g}}))\ne \emptyset$.
Denote $\gamma^u_i=f_1^{-k_i}\gamma_{p_{\hat{g}}}\cap (W^u_{\hat{\vep}}(p_{\hat{g}}) \setminus f_1^{-1} W^u_{\hat{\vep}}(p_{\hat{g}}))$. 
In the unstable manifold $W^u(p_{\hat{g}})$,
there exists $\lambda >1$ such that
$$ d(f_1^{t_i}r_i,f_1^{-\kappa_1+t_i}y)\le \eta_0\lambda^{t_i} \to 0,\quad \text{as} \,\,i\to +\infty, \forall \,\, y\in \gamma_{\hat{g}}. $$
Note that $f_1^{t_i} r_i\in W^u_{\hat{\vep}}(p_{\hat{g}}) \setminus f_1^{-1} W^u_{\hat{\vep}}(p_{\hat{g}})$. Thus $f_1^{-\kappa_1+t_i}y$ is  close to $W^u_{\hat{\vep}}(p_{\hat{g}}) \setminus f_1^{-1} W^u_{\hat{\vep}}(p_{\hat{g}})$ as we want. It follows that  $f_1^{-\kappa_1+t_i}y\in W^u_{\hat{\vep}}(p_{\hat{g}}) \setminus f_1^{-1} W^u_{\hat{\vep}}(p_{\hat{g}})$ or $f_1^{-\kappa_1+t_i-1}y\in W^u_{\hat{\vep}}(p_{\hat{g}}) \setminus f_1^{-1} W^u_{\hat{\vep}}(p_{\hat{g}})$, i.e.,     $|\kappa_1-t_i-k_i|\le1.$  
Further, there exists $m_i>0$ such that $\gamma_i=f_1^{k_i}(\gamma^u_i)\cap f_1^{-m_i}(W^s_{\hat{\vep}}(p_{\hat{g}}) \setminus f_2 W^s_{\hat{\vep}}(p_{\hat{g}}))$ is nonempty and contained in $\gamma_{p_{\hat{g}}}$. 
The order $n_i$ of $\gamma_i$ satisfies $n_i\ge k_i$ and $|\kappa_1-t_i-k_i|\le1$. 
Let $l_{i}$ be the length of $\gamma_{i}$.

Let $C$ be the center of $\gamma_i$. Define a bump function $\phi_2$ by: 
\begin{equation*}\phi_{2}(w)=\begin{cases}w \quad \text{for} \quad  w\notin B_{\frac{l_i}{2}}(C), \\
   \text{graph of } x^{2r+1}\sin(\frac{1}{x}) \text{ in  } (-\frac{l_{i}}{2},\frac{l_i}{2}) \text{ for } w\in \gamma_i. \end{cases} \end{equation*}
Then there exists a $C^{r}$-perturbation $f_1\longmapsto _{\frac{1}{4}\delta ,r}f_2$ such that $f_{2}$ has infinitely many transverse intersections between $W^{u}({p_{\hat{g}}},f_2)$ and $W^{s}({p_{\hat{g}}},f_2)$ in $B_{\frac{l_i}{2}}(C).$
Denote the set of these transverse homoclinic points by $T_i$.
The points in the set $T_i$ are also the intersections between $\gamma_i \subset W^{s}({p_{\hat{g}}},f_2)$ and $\phi_2 \gamma_i \subset W^{u}({p_{\hat{g}}},f_2)$.

Since $\phi_2$ is the identity outside $B_{\frac{l_i}{2}}(C)$, we have
$(f_2^{-1}\circ \phi_2) \gamma_i=f_1^{-1}\gamma_i$, $f_2 \gamma_i=f_1\gamma_i$ and $f_1^{-k_i}\gamma_i \subset W^{u}_{\hat{\varepsilon}}(p_{\hat{g}})\setminus f_2^{-1}W^{u}_{\hat{\varepsilon}}(p_{\hat{g}})$.
Thus
$$f_2^{-k_i} ( \phi_2  \gamma_i )= f_1^{-k_i} \gamma_i  \subset W^{u}_{\hat{\varepsilon}}(p_{\hat{g}})\setminus f_2^{-1}W^{u}_{\hat{\varepsilon}}(p_{\hat{g}}).$$
Then $f_2^{-k_i}T_i  \subset W^{u}_{\hat{\varepsilon}}(p_{\hat{g}})\setminus f_2^{-1}W^{u}_{\hat{\varepsilon}}(p_{\hat{g}})$.
Since $T_i\subset \gamma_i$ with $\gamma_i \subset f_1^{m_i}(W^{s}_{\hat{\varepsilon}}({p_{\hat{g}}})\setminus f_{2}W^{s}_{\hat{\varepsilon}}({p_{\hat{g}}}))$,
$f_2^{m_i}T_i  \subset W^{u}_{\hat{\varepsilon}}(p_{\hat{g}})\setminus f_2^{-1}W^{u}_{\hat{\varepsilon}}(p_{\hat{g}})$.
So $f_2^{-k_i}T_i \subset H({p_{\hat{g}}},n_{i},\hat{\varepsilon})$.

For any $N>0$, we can find $n_i>N$ and $T_i$  such that $f_{2}^{-k_i} x \in H({p_{\hat{g}}},n_i,\hat{\varepsilon})$ for all $x\in T_i$. 
Since $T_i$ contains infinitely many transverse homoclinic points, we can choose at least $n_{i} a_{n_{i}}$ transverse homoclinic points in $H({p_{\hat{g}}},n_{i},\hat{\varepsilon})$, i.e.
\begin{align} \label{N_{1}}
    \sharp H({p_{\hat{g}}},n_{i},\hat{\varepsilon})\ge n_{i} a_{n_{i}}.
\end{align}

Now we show that $f_0$ may be $C^{r}$ $\delta$-perturbed into an $C^r$ open set $U^r_{f_0,N, \delta}\subset \mathcal{N}^r$ such that for any $f\in U^r_{f_0,N,\delta} $, the inequality (\ref{N_{1}}) holds.
 Note that $\{n_{i}\}_{i\ge 1}=\{n_{i}^{(\delta)}\}_{i\ge 1}$ depends on $\delta$.   
 
\begin{figure}[h]
    \centering
    \includegraphics[width=0.8\linewidth]{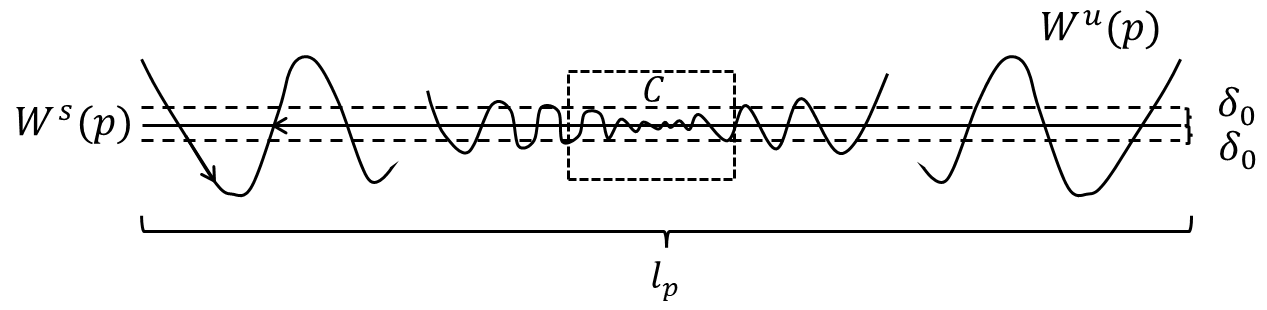}
    \caption{By a small $C^r$-perturbation, the inequality (\ref{N_{1}}) still holds.}
    \label{hold1}
\end{figure}

\textit{\textbf{ Step 5.}} Construct  $C^r$ residual set of $\mathcal{N}^r$.

(1) $2\le r<\infty$.   Let $\delta_m=\frac1m$, and denote   $U^{r}_{\delta_m,N} =\underset{f_0}{\bigcup }  U^r_{f_0,N,\delta_m}$. 
By iterating this process, there exists a $C^r$ residual set $\mathcal{R} _{a}^r=\underset{N\ge 1}{\bigcap } \underset{\delta_m}{\bigcup} U^r_{\delta_m,N} =\underset{N\ge 1}{\bigcap }  \underset{\delta _m}{\bigcup }  \underset{f_0\in D_r}{\bigcup} U^r_{f_0,N,\delta_m}$ depending on the sequence $\{ a_{n} \} _{n=1}^{\infty }$ with the property that $f\in \mathcal{R}^r _{a}$ implies that for any $N>0$, we can find $U^r_{f_0,N, \delta_m}\ni f$ with $k_N:=n_{(f_0,N, \delta_m)}>N$ satisfying:
\begin{align}
    \sharp H(p_{\hat{g}},k_N,\varepsilon)\ge k_N  a_{k_N},
\end{align}
i.e.
\begin{align}
    \frac{1}{a_{k_N}} \sharp H(p_{\hat{g}},k_N,\varepsilon)\ge k_N.
\end{align}
This proves the result for $2\le r<\infty$.

(2)  $r=\infty$.    Let $f_0\in D_{\infty}$. For any $\delta>0$, and $2\le m<\infty$, by Step 1-4,   we can obtain  that $f_0$ can  be $C^{m}$ $\delta$-perturbed into an $C^m$ open set $U^m_{f_0,N, \delta}$ such that for any $f\in U^m_{f_0,N,\delta} $, the inequality (\ref{N_{1}}) holds.  Let $\delta_m=\frac1m$.  Denote $U^{\infty}_{\delta_m,N} =\underset{f_0\in D_{\infty}}{\bigcup }  (U^m_{f_0,N,\delta_m}\cap \Diff^{\infty}(M))$. It is obvious that $U^{\infty}_{\delta_m,N}$  is a $C^{\infty}$ open  subset of $\mathcal{N}^{\infty}$. We show that it is also $C^{\infty}$  dense in $\mathcal{N}^{\infty}$.   
Denote $C=\sum_{2}^{\infty} \frac{1}{(1+r)^2}$ and $b_m=\sum_{m+1}^{\infty} \frac{1}{(1+r)^2}$.  Then \begin{equation}\label{R}
    \sum_{r=2}^{m } \frac{\delta_m}{(1+r)^{2}}=\delta_m\sum_{r=2}^{m } \frac{1}{(1+r)^{2}}\le C\delta_m,
\end{equation}
By the definition of $C^{\infty}$ distance (\ref{distance}),  we have $U^m_{N,\delta_m} $ is $(C\delta_m+b_m)$-dense in $\mathcal{N}^{\infty}$ with respect to $d_{\infty}$.  Thus $U^{\infty}_{N} =\underset{m}{\bigcup }  (U^m_{\delta_m, N}\cap \Diff^{\infty}(M))$ is $C^{\infty}$ dense in $\mathcal{N}^{\infty}$.  It follows that 
$$\mathcal{R} _{a}^{\infty}=\underset{N\ge 1}{\bigcap } U^{\infty}_{N} $$
 is a residual subset of $\mathcal{N}^{\infty}$, satisfying Theorem \ref{theorem c} for $r=\infty$.
\end{proof}

\end{document}